\documentclass[11pt]{article}

\usepackage{sectsty}
\usepackage{graphicx}
\usepackage{titling}
\usepackage{authblk}
\usepackage{titlesec}
\usepackage[dvipsnames]{xcolor}

\usepackage{dsfont}
\usepackage{amsmath}               
\usepackage{amsfonts}              
\usepackage{amsthm}                
\usepackage{bm}
\usepackage{stmaryrd}
\usepackage{float}
\usepackage{subcaption}
\usepackage{mathtools}
\usepackage{comment}

\newtheorem{proposition}{Proposition}[section]
\newtheorem{lemma}{Lemma}[section]

\titleformat*{\section}{\Large\bfseries}
\titleformat*{\subsection}{\large\bfseries}
\titleformat*{\subsubsection}{\large\bfseries}
\titleformat*{\paragraph}{\large\bfseries}
\titleformat*{\subparagraph}{\large\bfseries}

\usepackage{tikz}

\newcommand{\tikzcmark}{%
\tikz[scale=0.23] {
    \draw[line width=1.5,line cap=round] (0.25,0) to [bend left=10] (1,1);
    \draw[line width=1.5,line cap=round] (0,0.35) to [bend right=1] (0.23,0);
}}

\title{ Structure-preserving finite element approximations of a hybrid relativistic cold fluid - particle model}
\author[1]{Tileuzhan Mukhamet}
\author[1]{Katharina Kormann}
\affil[1]{Faculty of Mathematics, Ruhr-University Bochum, Germany}
\date{\today}

\titleformat{\section}  
  {\fontsize{12}{12}\bfseries} 
  {\thesection} 
  {0.5em} 
  {} 
  [] 

\begin{document}
\maketitle	

\begin{abstract}
We derive mixed finite element discretizations of a cold relativistics fluid model from approximations of the Poisson bracket that preserve mass, energy and the divergence constraints. For time-discretization we derive an implicit energy-conserving average-vector field method or apply an explicit strong-stability preserving Runge--Kutta scheme. We also consider a coupling of the fluid model to relativistic particles. We perform a numerical study of the scheme which shows convergence and conservation properties of the proposed methods and apply the new scheme to a plasma wake field simulation.
\end{abstract}




\section{Introduction}

In this paper, we construct a structure preserving finite element discretization of the hybrid cold relativistic fluid model. Fluid models provide collective fluid-like description of plasma and are of interest when micro-scale kinetic effects can be neglected. Among them are the well-known magnetohydrodynamic (MHD) models as well the moment-based models \cite{grad1949kinetic, gradshyperbolicity}. Cold models are relatively basic fluid models that assume zero temperature and are relevant when pressure effects are negligible. These models are used for the construction of dispersion relations, to study wave phenomena, and dynamics of plasma near equilibrium. They also serve as a cheaper alternative to the more expensive particle-in-cell simulations. Cold models can be described as relativistic or non-relativistic, linear or non-linear, models with constant density or variable density. 

The model we consider here is the relativistic non-linear model with variable density that
consists of fluids and also particles that are coupled through the Maxwell's equations. Cold models such as this find its application in laser-plasma interaction and particle acceleration studies to simulate the evolution of wake fields generated by strong laser pulses and plasma beams \cite{bera2015fluid}. Finite volume discretization of the model was considered recently in \cite{warpx}.

The unknowns of the model are respectively the fluid density $\rho(\bm x, t) \in \mathbb R$, the fluid momentum $\bm M(\bm x, t) \in \mathbb R^3$, particles' positions $\bm X_k(t) \in \mathbb R^3$, particles' momentum $\bm U_k(t) \in \mathbb R^3$, the electric field $\bm E(\bm x, t) \in \mathbb R^3$, and the magnetic field $\bm B(\bm x, t) \in \mathbb R^3$. 
The fluid part includes the first two-moments of the relativistic Vlasov-Maxwell model, namely the density and momentum. The governing equations for the fluid density and the fluid momentum in Gauss units are given by
\begin{align} 
\label{eq:deg_zero_hybrid_rho_relativ}
\partial_t \rho &= - \nabla \cdot \left( \frac{\bm M}{\gamma(m \bm M/ \rho)} \right) \\
\label{eq:deg_zero_hybrid_Mk_relativ}
\partial_t \bm M &= -  \nabla \cdot \left( \bm M \otimes \frac{\bm M}{\rho \gamma(m \bm M/ \rho)} \right) + \rho \frac{e}{m} \left[ \bm E + \frac{\bm M}{\rho \, c \, \gamma(m \bm M/ \rho)} \times \bm B \right] \\ \label{eq:gamma}
 & \text{where }\gamma(\bm u)  = \sqrt{1 + \frac{\bm u \cdot \bm u}{m^2 c^2}}
\end{align}
the equations for the particles' positions and momentum are given by
\begin{align}
\label{eq:deg_zero_hybrid_xk_relativ}
&\partial_t \bm X_k = \frac{\bm U_k}{m \gamma(\bm U_k)}\\ \label{eq:deg_zero_hybrid_uk_relativ}
&\partial_t \bm U_k = e \left( \bm E(\bm X_k, t)+ \frac{\bm U_k \times \bm B(\bm X_k, t)}{cm \gamma(\bm U_k)} \right)
\end{align}
and the Maxwell's equations are given by
\begin{align} \label{eq:maxwell_e}
& -\frac{1}{c} \partial_t \bm E + \nabla \times \bm B = \frac{4 \pi}{c} \bm J \\ \label{eq:maxwell_b}
& \frac{1}{c} \partial_t \bm B + \nabla \times \bm E = 0 \\ \label{eq:maxwell_div_e}
& \nabla \cdot \bm E = 4 \pi \, e \, \left( \frac{\rho}{m}  + \sum_k^N w_k \, \delta(\bm x - \bm X_k) \frac{\bm U_k}{\gamma(\bm U_k)} - n_0 \right) \\ \label{eq:maxwell_div_b}
& \nabla \cdot \bm B  = 0
\end{align}
where $c$ is the speed of light, $m$ is the mass of an ion, $e$ is the charge of an ion,  and $n_0$ denotes the constant background density

The current density $\bm J(\bm x, t)$, that couples fluids and particles to the Maxwell's equations, is defined as follows
\begin{align}
\bm J = \frac{e}{m} \frac{\bm M}{\gamma(m \bm M / \rho)} + \frac{e}{m} \sum_k^N w_k \delta(\bm x - \bm X_k) \frac{\bm U_k}{\gamma(\bm U_k)}
\end{align}
where $w_k$ is the weight carried by the $k$-th particle; and $\delta(\bm x - \bm X_k)$ is the Dirac delta distribution centered at the particle's position. The presence of the Dirac delta distributions implies that the Maxwell's equations (\ref{eq:maxwell_e}-\ref{eq:maxwell_div_b}) are defined in the sense of distributions. It is therefore natural to use weak formulations when defining Maxwell’s equations at the discrete level.

In this work, we construct a finite-element based discretization that preserves the following invariants and constraints
\begin{align} \label{eq:mass_cont}
& \text{Total mass:      } \frac{1}{m} \int_{\Omega} \rho \, dx + \sum_k^N w_k\\ \label{eq:energy_cont}
& \text{Total energy:    } \int_{\Omega} \rho (\gamma(m \bm M/ \rho)-1) c^2 dx + \sum_k^N w_k (\gamma(\bm U_k)-1)m c^2 +\frac{1}{8 \pi} \int_{\Omega} \left( \bm E^2 + \bm B^2 \right) dx \\ \label{eq:gauss_cont}
& \text{weak Gauss law:  } \int_{\Omega} \nabla_w \cdot \bm E \, \phi \, dx = 4 \pi e  \int_{\Omega} \left(\frac{\rho}{m} + \sum_k^N w_k \delta(\bm x - \bm X_k) - n_0 \right) \, \phi \, dx \quad \forall \phi \in V \\ \label{eq:divB_cont} 
& \text{divB constraint: } \nabla \cdot \bm B = 0
\end{align}
at spatially and temporally discrete level. 
In the Gauss law equation \eqref{eq:gauss_cont}, $\nabla_w$ is a weak gradient operator and $\phi$ are test functions from a suitable finite element space $V$. The details are discussed in the upcoming Section \ref{sec:spatial_discretization}.

These conservation properties hold on bounded domains $\Omega \in \mathbb R^3$, subject to the boundary constraints $\bm M \cdot \bm n = 0$, $\bm n \times \bm E=0$, $\bm n \cdot \bm B=0$ on $\partial \Omega$.

In the past, several authors considered related models. Concerning fluid models, structure preserving discretization of the incompressible Euler equations, that includes the momentum and density equations, was developed in \cite{gawlik2020conservative}. Their finite element scheme preserves total mass, total energy, and total squared density at the spatially and temporally discrete levels. The method in \cite{gawlik2020conservative} is rooted in the variational discretization in \cite{gawlik2021variationallie}. The method was later extended to MHD in \cite{gawlik2022finite}. Our discretization of the fluid part (\ref{eq:deg_zero_hybrid_rho_relativ}-\ref{eq:deg_zero_hybrid_Mk_relativ}) resembles the discretization in \cite{gawlik2020conservative} but is not the same. The relativistic equations has stronger nonlinearity, which makes structure preservation at the fully discrete level more challenging. We resolve non-linearities with the average-vector field gradient \cite{average_vector} and preserve the invariants at the fully-discrete level. Another feature of our discretization is that we use the Poisson structure. Structure preserving discretization of the model with Poisson structure was developed recently for the non-relativistic linearized cold model in \cite{moralsancoz}. Exploiting the Poisson structure allows for a discretization with an antisymmetric form, that preserves energy and prevents numerical heating. This is not the only approach, however, there are other methods that lead to energy invariance \cite{gawlik2021structure, CARLIER2025113647}.  

Structure-preserving particle-in-cell discretization based on the Finite Element Exterior Calculus framework was developed in \cite{gempic}. This particle method exactly preserves mass, Gauss law, and divB constraint. At the semi-discrete level it also preserves energy; while at the fully discrete level, the energy is bounded but not conserved. Exact energy conserving time-propagator for the structure-preserving particle-in-cell method was later developed in \cite{kormann_time} using the discrete gradient method \cite{discretegradient, average_vector}.

Models that consider both particles and fluids are known as hybrid models. These models are attractive because they offer a balance between the computational efficiency of fluid models and the accuracy of particle-based approaches. In reference \cite{holderied2021mhd}, a structure-preserving discretization was derived for the linearized MHD equations with energetic particles. We note that the model (\ref{eq:deg_zero_hybrid_rho_relativ}-\ref{eq:maxwell_div_b}) is different from the linear model considered in \cite{holderied2021mhd}.  To the best of our knowledge the   structure-preserving discretization of the hybrid model (\ref{eq:deg_zero_hybrid_rho_relativ}-\ref{eq:maxwell_div_b}) has not been previously studied in the literature.  The only other discretization of the current model is due to \cite{warpx}.

Section 2 introduces the Poisson structure of the continuum model. Structure-preserving spatial discretizations based on two choices for the basis functions are proposed in Section 3. In Section 4, an implicit time-integration scheme is developed keeping the original invariants. Explicit strong-stability preserving Runge-Kutta methods are also considered. The latter one, in the absence of particles preserves exactly total mass, Gauss's law, and divB constraint; while the energy error is of the order of the time-integration error $\mathcal O \left( \Delta t^p \right)$. In the presence of particles, Gauss's law is not preserved and we propose cleaning to contain it. In Section 5, we study the schemes' properties numerically and show application to plasma wake simulation.     

\section{Hamiltonian structure and weak formulation}

\subsection{Poisson bracket of the cold relativistic plasma model}
The model (\ref{eq:deg_zero_hybrid_rho_relativ}-\ref{eq:maxwell_div_b}) is a conservative system describing collisionless plasma. This model admits a so called Hamiltonian structure that consists of a Poisson bracket and a Hamiltonian. A reader not familiar with these terms can consult for example \cite{morrison1998hamiltonian}. 

Here, we use only basic facts about the Hamiltonian structure to construct our conservative weak formulations. First we introduce a Poisson bracket $\{\cdot,\cdot\}$.  Let $f(\bm x,t)$ be a function of position and time and $\mathcal F[f](t)$ be a functional that maps functions of position and time to  functions of time only. The bracket allows to expresses the time-rate of the functional as follows:
\begin{align} \label{eq:dt_F}
\partial_t \mathcal F = \{ \mathcal F, \mathcal H \}
\end{align} 
where $\mathcal H[f](t)$ is the Hamiltonian functional (total energy).
The model (\ref{eq:deg_zero_hybrid_rho_relativ}-\ref{eq:maxwell_b}) consists of functions $\rho, \bm M, \bm B, \bm E$ that depend on position and time and functions $\bm X_k, \bm U_k$ that depend on time. We can view $\bm X_k$ and $\bm U_k$ as the functions of position as well by multiplying them with the unit function $\mathds 1 (\bm x)$. 

We consider the following Poisson bracket
\begin{align}  \notag
\{ \mathcal F, \mathcal G \} & = \int_{\Omega} \left[ \frac{\delta \mathcal{G}}{ \delta \bm M} \cdot \nabla \frac{\delta \mathcal{F}}{\delta \bm M} -\frac{\delta \mathcal{F}}{ \delta \bm M} \cdot \nabla \frac{\delta \mathcal{G}}{\delta \bm M}\right] \cdot \bm M \, dx + \int_{\Omega} \left[ \frac{\delta \mathcal{G}}{\delta \bm M} \cdot \nabla \frac{\delta \mathcal{F} }{\delta \rho} - \frac{\delta \mathcal{F}}{\delta \bm M} \cdot \nabla \frac{\delta \mathcal{G} }{\delta \rho}  \right] \, \rho \, dx\\ \label{eq:degree_zero_hybrid_bracket_particles}
&+ \frac{4 \pi e}{m} \int_{\Omega} \left[ \frac{\delta \mathcal{F}}{ \delta \bm M} \cdot \frac{\delta \mathcal{G} }{\delta \bm E}-\frac{\delta \mathcal{G} }{\delta \bm M} \cdot \frac{\delta \mathcal{F}}{\delta \bm E}   \right] \, \rho \, dx + \frac{e}{mc} \int_{\Omega} \rho \bm B \cdot \left( \frac{\delta \mathcal{F}}{\delta \bm M} \times \frac{\delta \mathcal{G}}{\delta \bm M} \right) \, dx\\ \notag
&+ \sum_k^N \frac{1}{w_k} \left[ \frac{\partial \mathcal F}{\partial \bm X_k} \cdot \frac{\partial \mathcal G}{\partial \bm U_k} - \frac{\partial \mathcal G}{\partial \bm X_k} \cdot \frac{\partial \mathcal F}{\partial \bm U_k} \right] + 4 \pi \, e \sum_k^N \left. \left[ \frac{\partial \mathcal F}{\partial \bm U_k} \cdot \frac{\delta \mathcal G}{\delta \bm E} -  \frac{\partial \mathcal G}{\partial \bm U_k} \cdot \frac{\delta \mathcal F}{\delta \bm E} \right] \right\vert_{(\bm X_k, \bm U_k)}\\ \notag
& + \frac{e}{c} \sum_k^N \frac{1}{w_k} \bm B(\bm X_k, t) \cdot \left( \frac{\partial \mathcal F}{\partial \bm U_k} \times \frac{\partial \mathcal G}{\partial \bm U_k} \right) 
+ 4 \pi c \int_{\Omega} \left[ \frac{\delta \mathcal G}{\delta \bm B} \cdot \left( \nabla \times \frac{\delta \mathcal F}{\delta \bm E} \right)  -  \frac{\delta \mathcal F}{\delta \bm B} \cdot \left( \nabla \times \frac{\delta \mathcal G}{\delta \bm E} \right) \right] dx
\end{align}
where $\mathcal F$ and $\mathcal G$ are functionals of $(\rho, \bm M, \mathds 1 (\bm x) \bm X_k, \mathds 1(\bm x) \bm U_k,\bm E, \bm B )$. This bracket can be derived from the kinetic-multifluid bracket in \cite[equation (45)]{tronci2010}. To this end, one assumes a particle-like distribution for the energetic component $f=\sum_p^N w_p \, \delta(\bm x - \bm X_p) \, \delta(\bm u - \bm U_p)$ in \cite[equation (45)]{tronci2010} and replaces functional derivatives with function derivatives using the relations in \cite[equation (4.19a,b)]{gempic}.  In the bracket \eqref{eq:degree_zero_hybrid_bracket_particles}, the operator $\delta \mathcal F /\delta \cdot$ takes functional derivatives 
\begin{align}
\frac{d}{d \epsilon} \left. \mathcal F[\bm M +\epsilon \delta \bm M] \right\vert_{\epsilon =0} = \int_{\Omega} \frac{\delta \mathcal F}{\delta \bm M} \cdot  \delta \bm M \, dx
\end{align}

The functionals  $\mathcal F$ and $\mathcal G$ can be considered as functionals of $(\rho, \bm M,\bm E, \bm B )$ and functions of $(\bm X_k, \bm U_k)$. The operator $\partial \mathcal F / \partial  \cdot $ takes regular partial derivatives with respect to $\bm X_k$ and $\bm U_k$, evaluated at a fixed point in time.

The Hamiltonian, which is also the total energy of the system, is given by
\begin{align} \label{eq:relativistic_hamiltonian_hybrid_degree_zero_conserved}
\mathcal H = \int_{\Omega} \rho (\gamma(m \bm M/ \rho)-1) c^2 dx + \sum_k^N w_k (\gamma(\bm U_k)-1)m c^2 +\frac{1}{8 \pi} \int_{\Omega} \left( \bm E^2 + \bm B^2 \right) dx
\end{align}
its functional/function derivatives are
\begin{align} \label{eq:functional_der_rel_hamil} 
& \frac{\delta \mathcal H}{\delta \rho}= \left[ \gamma(m \bm M/ \rho)-1-\frac{\bm M \cdot \bm M}{\rho^2 \, c^2 \, \gamma(m \bm M/ \rho)} \right] c^2,  \hspace{0.5cm} \frac{\delta \mathcal H}{\delta \bm M} = \frac{\bm M}{\rho \gamma(m \bm M/ \rho)} \\ \notag & \frac{\delta \mathcal H}{\delta \bm E}=\frac{\bm E}{4 \pi} ,
\hspace{0.5cm} \frac{\delta \mathcal H}{\delta \bm B}=\frac{\bm B}{4 \pi}, \hspace{0.5cm} \frac{\partial \mathcal H}{\partial \bm U_k}=\frac{w_k \, \bm U_k}{m \, \gamma(\bm U_k)}
\end{align}

\subsection{Weak equations of motion}
In this section, we derive weak formulations for the system (\ref{eq:deg_zero_hybrid_rho_relativ}-\ref{eq:maxwell_b}) using the bracket \eqref{eq:degree_zero_hybrid_bracket_particles} and the Hamiltonian \eqref{eq:relativistic_hamiltonian_hybrid_degree_zero_conserved}.

To this end, we assume the variables $\rho, \bm M, \bm E, \bm B$ are smooth functions and we test them with smooth time-independent test functions $\phi(\bm x), \bm \mu(\bm x), \bm \nu(\bm x), \bm \tau(\bm x)$ 
\begin{align} \label{eq:functionals_for_weak_form}
\int_{\Omega} \rho \, \phi \, dx \hspace{0.5cm} \int_{\Omega} \bm M \cdot \bm \mu \, dx \hspace{0.5cm} \int_{\Omega} \bm E \cdot \bm \nu \, dx \hspace{0.5cm} \int_{\Omega} \bm B \cdot \bm \tau \, dx 
\end{align}
The integrals in \eqref{eq:functionals_for_weak_form} are functionals of $\rho$, $\bm M$, $\bm E$, and $\bm B$, respectively. We substitute these functionals for $\mathcal F$ in \eqref{eq:dt_F} and derive weak formulations. To this end, we need to evaluate functional derivatives that appear inside the bracket \eqref{eq:degree_zero_hybrid_bracket_particles}. Most of the functional derivatives are zero; the functional derivatives of interest are
\begin{align}
\frac{\delta \int_{\Omega} \rho \, \phi \, dx}{\delta \rho} = \phi \hspace{0.5cm} \frac{\delta \int_{\Omega} \bm M \cdot \bm \mu  \, dx}{\delta \bm M} = \bm \mu \hspace{0.5cm} \frac{\delta \int_{\Omega}  \bm E \cdot \bm \nu  \, dx}{\delta \bm E} = \bm \nu \hspace{0.5cm} \frac{\delta \int_{\Omega}  \bm B \cdot \bm \tau \, dx}{\delta \bm B} = \bm \tau 
\end{align}
The functional derivatives of the Hamiltonian $\mathcal H$ were evaluated previously in \eqref{eq:functional_der_rel_hamil}. Using these expressions we obtain the following weak formulations of the equations of motion of the dynamic variables
\begin{align} \label{eq:rel_weak_form_rho_original}
& \int_{\Omega} \partial_t \rho \, \phi = \int_{\Omega} \rho \frac{\bm M}{\rho \gamma(m \bm M/ \rho)} \cdot \nabla \phi \\
& \int_{\Omega} \partial_t \bm M \cdot \bm \mu \, dx = \int_{\Omega} \left[ \frac{\bm M}{\rho \gamma(m \bm M/ \rho)} \cdot \nabla \bm \mu -\bm \mu \cdot \nabla \left( \frac{\bm M}{\rho \gamma(m \bm M/ \rho)} \right) \right] \cdot \bm M \, dx \\ \notag
& \hspace{2.5cm} - \int_{\Omega} \rho \, \bm \mu \cdot \nabla \left( \gamma(m \bm M/ \rho) -1- \frac{\bm M^2}{\rho^2 c^2 \gamma(m \bm M/ \rho)} \right) c^2 \, dx \\ \notag
& \hspace{2.5cm}  + \frac{e}{m} \int_{\Omega} \rho \left( \bm E + \frac{\bm M}{\rho \gamma(m \bm M/ \rho)} \times \frac{\bm B}{c} \right)  \cdot \bm \mu \, dx \\ \label{eq:rel_weak_form_E_original}
& \int_{\Omega} \partial_t \bm E \cdot \bm \nu \, dx = c \int_{\Omega} \bm B \cdot (\nabla \times \bm \nu) \, dx - \frac{4 \pi e }{m} \int_{\Omega}  \left( \rho \frac{\bm M}{\rho \gamma(m \bm M/ \rho) } \cdot \bm \nu \right) \, dx \\ \notag
& \hspace{2.5cm} - \frac{4 \pi e}{m} \sum_k^{N} \frac{w_k \bm U_k}{\gamma(\bm U_k)}  \cdot \bm \nu(\bm X_k) \, dx 
\\ \label{eq:rel_weak_form_B_original}
& \int_{\Omega} \partial_t \bm B \cdot \bm \tau \, dx = - c \int_{\Omega} \bm \tau \cdot (\nabla \times \bm E) \, dx 
\end{align}
Next, we consider the following functionals for the particles' position and momentum
\begin{align} \label{eq:pos_mom_functionals}
\int_{\Omega} \frac{\mathds 1(\bm x)}{|\Omega|} \bm X_k \, dx \hspace{1cm} \int_{\Omega} \frac{\mathds 1(\bm x)}{|\Omega|} \bm U_k \, dx
\end{align}
To evaluate the partial derivatives of the position and momentum functionals \eqref{eq:pos_mom_functionals}, we note that these functionals are equal to $\bm X_k$ and $\bm U_k$, respectively. We denote the $i$-th components of the position and momentum functions of the $l$-th particle by $(X_l)_i$ and $(U_l)_i$. The partial derivatives are given by
\begin{align} \notag
\frac{\partial  (X_l)_i}{\partial (X_k)_j} = \delta_{lk} \delta_{ij} \hspace{1cm} \frac{\partial  (U_l)_i}{\partial (U_k)_j} = \delta_{lk} \delta_{ij}
\end{align}
where $\delta_{ij}$ is the Kronecker delta that is equal to 1 for $i=j$ and $0$ otherwise.
Substituting $\bm X_k$ and $\bm U_k$ into the bracket \eqref{eq:degree_zero_hybrid_bracket_particles} together with the Hamiltonian \eqref{eq:relativistic_hamiltonian_hybrid_degree_zero_conserved}, and its partial derivatives in \eqref{eq:functional_der_rel_hamil}, the particle equations of motions read
\begin{align} \label{eq:particle_X_k}
&\partial_t \bm X_k(t) = \frac{\bm U_k}{m \gamma(\bm U_k)}\\ \label{eq:particle_U_k}
&\partial_t \bm U_k(t) = e \left( \bm E(\bm X_k, t)+ \frac{\bm U_k \times \bm B(\bm X_k, t)}{cm \gamma(\bm U_k)} \right) 
\end{align}

For the ease of notation, let us  introduce the relativistic velocity variable $\bm w = \frac{\bm M}{\rho \gamma(m \bm M/ \rho)}$. By studying eigenvalues of the advection operator of (\ref{eq:deg_zero_hybrid_rho_relativ}-\ref{eq:maxwell_b}), it is shown in Appendix \ref{appendix_A} that this quantity is the eigenvalue of the system and represents its speed. Our weak formulation reads:
\begin{align} \label{eq:weak_rho}
& \int_{\Omega} \partial_t \rho \, \phi = \int_{\Omega} \rho \, \bm w \cdot \nabla \phi \, dx \\ \label{eq:weak_M}
& \int_{\Omega} \partial_t \bm M \cdot \bm \mu \, dx = \int_{\Omega} \left( \bm w \cdot \nabla \bm \mu -\bm \mu \cdot \nabla \bm w \right) \cdot \bm M \, dx \\ \notag
& \hspace{2.5cm} - \int_{\Omega} \rho \, \bm \mu \cdot \nabla \left( \gamma(m \bm M/ \rho) -1- \frac{\bm M^2}{\rho^2 c^2 \gamma(m \bm M/ \rho)} \right) c^2 \, dx \\ \notag
& \hspace{2.5cm}  + \frac{e}{m}  \int_{\Omega} \rho \left( \bm E + \bm w \times \frac{\bm B}{c} \right)  \cdot \bm \mu \, dx \\ \label{eq:weak_E}
& \int_{\Omega} \partial_t \bm E \cdot \bm \nu \, dx = c \int_{\Omega} \bm B \cdot (\nabla \times \bm \nu) \, dx - \frac{4 \pi e }{m} \int_{\Omega}  \left( \rho \bm w \cdot \bm \nu \right) \, dx \\ \notag
& \hspace{2.5cm} - \frac{4 \pi e}{m} \sum_k^{N} \frac{w_k \bm U_k}{\gamma(\bm U_k)}  \cdot \bm \nu(\bm X_k) \, dx 
\\ \label{eq:weak_B}
& \int_{\Omega} \partial_t \bm B \cdot \bm \tau \, dx = - c \int_{\Omega} \bm \tau \cdot (\nabla \times \bm E) \, dx  \\
& \text{ with }\bm w = \frac{\bm M}{\rho \gamma(m \bm M/ \rho)}
\end{align}
\begin{align} \label{eq:X_k}
&\partial_t \bm X_k(t) = \frac{\bm U_k}{m \gamma(\bm U_k)} \hspace{0.5cm} \text{with } \bm X_k \in \Omega_h \\ \label{eq:U_k}
&\partial_t \bm U_k(t) = e \left( \bm E(\bm X_k, t)+ \frac{\bm U_k \times \bm B(\bm X_k, t)}{cm \gamma(\bm U_k)} \right) 
\end{align}

The weak formulation above has a peculiar momentum equation \eqref{eq:weak_M}. An arduous derivation in Appendix \ref{appendix_B} shows that this formulation can be derived starting from the original momentum equation \eqref{eq:deg_zero_hybrid_Mk_relativ} as well.

If we set $\phi=1$ in the density equation \eqref{eq:weak_rho}, we deduce conservation of total mass
\begin{align}
\int_{\Omega} \partial_t \rho \, dx= 0
\end{align}
If we take $\phi = \left( \gamma(m \bm M/ \rho) -1- \frac{\bm M^2}{\rho^2 c^2 \gamma(m \bm M/ \rho)} \right) c^2$ in the density equation \eqref{eq:weak_rho}, $\bm \mu= \bm w$ in the momentum equation \eqref{eq:weak_M}, $\bm \nu=\bm E$ in Ampere's equation \eqref{eq:weak_E}, and $\bm \tau = \bm B$ in Faraday's equation \eqref{eq:weak_B}, we deduce conservation of total energy
\begin{align} \label{eq:energy_rate}
\frac{d}{dt} \mathcal H &= \int_{\Omega} \partial_t \rho \left( \gamma(m \bm M/\rho)-1 -\frac{\bm M^2}{\rho^2 \, c^2 \, \gamma(m \bm M/\rho)} \right) c^2 \, dx + \int_{\Omega} \partial_t \bm M \cdot \bm w \, dx \\ \notag
&+ \frac{1}{4 \pi} \int_{\Omega} \left( \partial_t \bm E \cdot \bm E + \partial_t \bm B \cdot \bm B \right) \, dx  + \sum_k^N \frac{w_k \bm U_k}{m \gamma(\bm U_k)} \cdot \partial_t \bm U_k  = 0
\end{align}
We will use this basic property to construct structure preserving discretizations. 

\section{Spatial discretization} \label{sec:spatial_discretization}
In this section, we derive two mixed-finite element schemes for the system (\ref{eq:weak_rho}-\ref{eq:U_k}) based on two different choices of the spaces of approximation. 

Let us first introduce the notation.
We start by introducing the infinite dimensional spaces
\begin{align}
H^1(\Omega) &= \{f \in L^2(\Omega)\,\,\,  | \,\,\, \nabla f \in L^2(\Omega)^d \} \\
H(\text{div}, \Omega) &= \{ f \in L^2(\Omega)^d \,\,\, | \,\, \nabla \cdot f \in L^2(\Omega) \} \\
H(\text{curl}, \Omega) &= \{ f \in L^2(\Omega)^d \,\,\, | \,\, \nabla \times f \in L^2(\Omega)^d \} 
\end{align}
The key property of functions in the space $H(\text{div}, \Omega)$ is that their normal component is continuous accross any surface in domain. Similarly, for the functions in $H(\text{curl}, \Omega)$ the tangential component is continuous accross any surface.
We consider the following sub-spaces with homogeneous constraints on the boundary $\partial \Omega$
\begin{align} \label{eq:H1_hom}
\mathring {H}^1(\Omega) &= \{f \in H^1(\Omega)\,\,  | \,\, f=0 \text{ on } \partial \Omega \} \\ \label{eq:Hdiv_hom}
\mathring{H}(\text{div}, \Omega) &= \{ f \in H(\text{div}, \Omega) \,\,\, | \,\, n \cdot f = 0 \text{ on } \partial \Omega\} \\ \label{eq:Hcurl_hom}
\mathring{H}(\text{curl}, \Omega) &= \{ f \in H(\text{curl}, \Omega) \,\,\, | \,\, n \times f = 0 \text{ on } \partial \Omega\} \\ \label{eq:L2_hom}
\mathring{L}^2(\Omega) & = \{ f \in L^2(\Omega) \,\, | \int_{\Omega} f dx = 0\}
\end{align}
Let $\mathcal{T}_h$ denote a partitioning of $\Omega$ and let $\mathcal{E}_h$ denote the set of interior faces of $\mathcal{T}_h$. Let $K$ denote a cell in $\mathcal{T}_h$ and $e$ denote an edge in $\mathcal{E}_h$. The union of all cells in $\mathcal T_h$ defines the discrete domain $\Omega_h=\cup_{K \in \mathcal T_h} K$.

We consider several finite-dimensional spaces. For $k \geq 0$, we denote by $P_k(K)$ the space of polynomials of degree at most $k$ on cell $K \in \mathcal{T}_h$.  Moreover, we introduce
\begin{align} \label{eq:q_space_definition}
Q_{k} &= \{ f \in H^1(\Omega) \,\,\,  | \,\,\, f|_K  \in P_{k}(K), \,\, \forall K \in \mathcal{T}_h \} \\ 
 \label{eq:n_space_definition}
N_k &= \{ f \in H(\text{curl}, \Omega) \,\, | \,\, f|_K \in P_k(K)^d+ \bm x \times P_k(K)^d, \,\, \forall K \in \mathcal{T}_h \} 
\\ \label{eq:rt_space_definition}
RT_k &= \{ f \in H(\text{div}, \Omega) \,\, | \,\, f|_K \in P_k(K)^d+ \bm x P_k(K), \,\, \forall K \in \mathcal{T}_h \} \\
\label{eq:dg_space_definition}
DG_k &= \{ f \in L^2(\Omega) \,\,\,  | \,\,\, f|_K  \in P_k(K), \,\, \forall K \in \mathcal{T}_h \} 
\end{align}
Finite dimensional sub-spaces with homogeneous constraints are then defined by $\mathring{Q}_k(\Omega) = Q_k(\Omega) \cap \mathring{H}^1(\Omega)$, $\mathring{RT}_k(\Omega) = RT_k(\Omega) \cap \mathring{H}(\text{div}, \Omega)$, $\mathring{N}_k(\Omega) = N_k(\Omega) \cap \mathring{H}(\text{curl}, \Omega)$, $\mathring{DG}_k(\Omega) = DG_k(\Omega) \cap \mathring{L}^2(\Omega)$. On bounded domains with homogeneous constraints these spaces form the exact sequence \cite[Section 4.5.5]{arnold2018finite}:
\begin{align} \label{eq:exact_sequence}
\mathring Q_{k+1} \xrightarrow[]{\nabla} \mathring N_k \xrightarrow[]{\nabla \times} \mathring {RT}_k \xrightarrow[]{\nabla \cdot} \mathring{DG}_k
\end{align}

The diagram \eqref{eq:exact_sequence}, encodes the following properties
\begin{align} \label{eq:prop_Qk}
& \forall  \phi_h \in \mathring{Q}_{k+1}: \, \nabla \phi_h \in \mathring{N}_k  \text{ and } \nabla \times \nabla \phi_h = 0 \\ \label{eq:prop_Nk}
& \forall  \bm \nu_h \in \mathring{N}_k: \, \nabla \times \bm \nu_h \in \mathring{RT}_k \text{ and } \nabla \cdot \nabla \times \bm \nu_h = 0 \hspace{0.5cm}  \\ \label{eq:prop_RTk}
& \forall \bm \tau_h \in \mathring{RT}_k : \nabla \cdot \bm \tau_h \in \mathring{DG}_k  \hspace{0.5cm} 
\end{align} 

We define the weak divergence operator $\nabla_w \cdot: \mathring N_{k} \rightarrow \mathring Q_{k+1} $ and the weak gradient operator $\nabla_w: \mathring{DG}_k \rightarrow \mathring{RT}_k$ : 
\begin{align} \label{eq:weak_divergence}
&\int_{\Omega} \nabla_w \cdot \bm E_h \, \phi_h \, dx = - \int_{\Omega} \bm E_h \cdot \, \nabla \phi_h \, dx \hspace{1cm} \forall \bm E_h \in \mathring{N}_k \text{ and } \forall \phi_h \in \mathring{Q}_{k+1} \\ \label{eq:weak_gradient}
&\int_{\Omega} \nabla_w \varphi_h \cdot \bm B_h \, dx = - \int_{\Omega} \nabla \cdot \bm B_h \, \varphi_h \, dx \hspace{1cm} \forall \bm B_h \in \mathring{RT}_k \text{ and } \forall \varphi_h \in \mathring{DG}_k
\end{align}

Next, for a function $g$, we define the jump and the average on the edge $e=K_1 \cap K_2 \in \mathcal{E}_h$
\begin{align}
\llbracket g \rrbracket = g_1 \bm n_1 + g_2 \bm n_2, \hspace{1cm} \{g\} = \frac{g_1+g_2}{2}
\end{align}
where $g_1$ and $g_2$ are function values on $K_1$ and $K_2$, $\bm n_1$ is the unit normal vector to $e$ pointing from $K_1$ to $K_2$, and $\bm n_2$ is the unit normal vector to $e$ pointing from $K_2$ to $K_1$, i.e. $\bm n_1=-\bm n_2$. For vector fields, we assign a normal vector to each edge and define the jump in the direction of this normal vector. For example, if $\bm n_1$ is selected as the normal vector of the edge $e$, then we define
\begin{align}
\llbracket \bm E \rrbracket = \bm E_1 - \bm E_2 \hspace{1cm} \{ \bm E\} = \frac{\bm E_1 + \bm E_2}{2}
\end{align}

\subsection{A flux-free formulation}

Our first formulation is based on the following choice of finite element spaces: $\rho_h \in Q_{k+1}$, $\bm M_h \in Q_{k+1}$, $\bm E_h \in \mathring{N}_k$, $\bm B_h \in \mathring{RT}_k$. This choice yields the following weak form of the equations (\ref{eq:weak_rho}-\ref{eq:U_k}):
\begin{align} \label{eq:qqnrt_rho}
& \int_{\Omega} \partial_t \rho_h \, \phi_h = \int_{\Omega} \rho_h \, \bm w_h \cdot \nabla \phi_h \, dx  \hspace{1cm} \forall \phi_h \in Q_{k+1} \\ \label{eq:qqnrt_M}
& \int_{\Omega} \partial_t \bm M_h \cdot \bm \mu_h \, dx = \int_{\Omega} \left( \bm w_h \cdot \nabla \bm \mu_h -\bm \mu_h \cdot \nabla \bm w_h \right) \cdot \bm M_h \, dx \\ \notag
& \hspace{2.5cm} - \int_{\Omega} \rho_h \, \bm \mu_h \cdot \nabla \left( \overline{ \gamma(m \bm M_h/ \rho_h) -1- \frac{\bm M_h \cdot \bm M_h}{\rho_h^2 c^2 \gamma(m \bm M_h/ \rho_h)} } \right) c^2  \, dx \\ \notag
& \hspace{2.5cm}  + \frac{e}{m} \int_{\Omega} \rho_h \left( \bm E_h + \bm w_h \times \frac{\bm B_h}{c} \right)  \cdot \bm \mu_h \, dx \hspace{1cm} \forall \bm \mu_h \in Q_{k+1}^d \\ \label{eq:qqnrt_E}
& \int_{\Omega} \partial_t \bm E_h \cdot \bm \nu_h \, dx = c \int_{\Omega} \bm B_h \cdot (\nabla \times \bm \nu_h) \, dx - \frac{4 \pi e }{m} \int_{\Omega}   \rho_h \bm w_h \cdot \bm \nu_h  \, dx \\ \notag
& \hspace{2.5cm} - \frac{4 \pi e}{m} \sum_k^{N} \frac{w_k \bm U_k}{\gamma(\bm U_k)}  \cdot \bm \nu_h(\bm X_k)  \hspace{1cm} \forall \bm \nu_h \in \mathring{N}_k
\\ \label{eq:qqnrt_B}
& \int_{\Omega} \partial_t \bm B_h \cdot \bm \tau_h \, dx = - c \int_{\Omega} \bm \tau_h \cdot (\nabla \times \bm E_h) \, dx \hspace{1cm} \forall \bm \tau_h \in \mathring{RT}_k \\ \label{eq:qqnrt_w}
& \text{ with }\bm w_h = \overline{ \frac{\bm M_h}{\rho_h \gamma(m \bm M_h/ \rho_h)}}
\end{align}
\begin{align} \label{eq:qqnrt_X_k}
&\partial_t \bm X_k(t) = \frac{\bm U_k}{m \gamma(\bm U_k)} \hspace{0.5cm} \text{for } \bm X_k \in \Omega_h\\ \label{eq:qqnrt_U_k}
&\partial_t \bm U_k(t) = e \left( \bm E_h(\bm X_k, t)+ \frac{\bm U_k \times \bm B_h(\bm X_k, t)}{cm \gamma(\bm U_k)} \right) 
\end{align}
where the bar in the momentum equation \eqref{eq:qqnrt_M} denotes the $L_2$ projection onto $Q_{k+1}$, while the bar in the velocity equation \eqref{eq:qqnrt_w} denotes the $L_2$ projection onto $Q_{k+1}^d$. 

\vspace{2.5mm}
\noindent \textbf{Conservation properties of the spatial discretization.}
Here we study the conservation properties of our spatial discretization. First, we need a Lemma that states the charge balance for the particles.
\begin{lemma} \label{prop:particle_mass_conservation}
The Particle-in-cell equation \eqref{eq:qqnrt_X_k}-\eqref{eq:qqnrt_U_k} satisfy the following charge balance equation
\begin{align}
\sum_k^{N} \int_{\Omega_h} \, \partial_t \left( w_k \delta(\bm x - \bm X_k) \right) \, \phi_h(\bm x) \, dx = \sum_k^{N} \frac{w_k \bm U_k}{m \, \gamma(\bm U_k)}  \cdot \nabla \phi_h(\bm X_k) \hspace{1cm} \forall \phi_h \in \mathring{Q}_{k+1}
\end{align}
\end{lemma}
\begin{proof}

\begin{align*}
\sum_k^{N} \int_{\Omega_h} \, \partial_t \left( w_k \delta(\bm x - \bm X_k) \right) \, \phi_h(\bm x) \, dx  
&\stackrel{\text{(chain rule)}}{=} \sum_k^{N} \int_{\Omega_h} w_k \, \partial_t \bm X_k \cdot \nabla \delta(\bm x - \bm X_k) \, \phi_h(\bm x) \, dx \\ \notag
 &\stackrel{\text{(integration by parts)}}{=} \sum_k^{N} \int_{\Omega_h} w_k \, \delta(\bm x - \bm X_k) \partial_t \bm X_k  \cdot \nabla \phi_h(\bm x) \, dx \\
&\stackrel{\eqref{eq:qqnrt_X_k}}{=}\sum_k^{N} \frac{w_k \bm U_k}{m \, \gamma(\bm U_k)}  \cdot \nabla \phi_h(\bm X_k) 
\end{align*}
\end{proof}

With this, we can prove the following conservation properties:

\begin{proposition} \label{prop:h1h1_proof}
The scheme (\ref{eq:qqnrt_rho}-\ref{eq:qqnrt_U_k}), conserves over time the value of the total fluid mass $\int_{\Omega_h} \rho \, dx$, the total particle mass $\sum_k^N w_k$, the total energy $\int_{\Omega_h} \rho_h (\gamma(m \bm M_h/ \rho_h)-1) c^2 dx + \sum_k^N w_k (\gamma(\bm U_k)-1)m c^2 +\frac{1}{8 \pi} \int_{\Omega_h} \left( \bm E_h^2 + \bm B_h^2 \right) dx$, weak Gauss' law $ \int_{\Omega_h} \left[ \nabla_w \cdot \bm E - 4 \pi e \left(\frac{\rho_h}{m} + \sum_k^N w_k \delta(\bm x - \bm X_k) \right) \right] \, \phi_h \, dx$ for all $\phi_h \in \mathring{Q}_{k+1}$, and $\nabla \cdot \bm B_h$ constraint.
\end{proposition}

\begin{proof}
We take $\phi_h = 1$ in \eqref{eq:qqnrt_rho} to deduce conservation of total fluid mass
\begin{align}
\int_{\Omega_h} \partial_t \rho_h \, dx = 0
\end{align}

The total particle mass $\sum_k^N w_k$ is conserved, since the weights of the  particles are constant and the particles are assumed to be in the domain.

To show the conservation of total energy, we use the chain rule to express the time-rate of the relativistic factor $\gamma$, that was defined in \eqref{eq:gamma},
\begin{align} \label{eq:gamma_rate_m}
\partial_t \gamma(m \bm M_h/ \rho_h) = \frac{\bm M_h}{\gamma(m \bm M_h/ \rho_h) \, \rho_h^2 \, c^2} \cdot \partial_t \bm M_h-\frac{\bm M_h^2}{\gamma(m \bm M_h/\rho_h) \, \rho_h^3 \, c^2} \, \partial_t \rho_h
\end{align}
and 
\begin{align} \label{eq:gamma_rate_u}
\partial_t \gamma(\bm U_k) = \frac{\bm U_k}{m^2 c^2 \gamma(\bm U_k)} \cdot \partial_t \bm U_k
\end{align} 
we can write the energy rate as
\begin{align} \label{eq:energy_rate}
\frac{d}{dt} \mathcal H &= \int_{\Omega_h} \partial_t \rho_h \left( \gamma(m \bm M_h/\rho_h)-1 -\frac{\bm M_h^2}{\rho_h^2 \, c^2 \, \gamma(m \bm M_h/\rho_h)} \right) c^2 \, dx + \int_{\Omega_h} \partial_t \bm M_h \cdot \bm w_h \, dx \\ \notag
&+ \frac{1}{4 \pi} \int_{\Omega_h} \left( \partial_t \bm E_h \cdot \bm E_h + \partial_t \bm B_h \cdot \bm B_h \right) \, dx  + \sum_k^N \frac{w_k \bm U_k}{m \gamma(\bm U_k)} \cdot \partial_t \bm U_k .
\end{align}
Since $\partial_t \rho_h \in Q_{k+1}$, by the definition of the  $L_2$ projection, this can be also written as
\begin{align} \label{eq:energy_rate}
\frac{d}{dt} \mathcal H &= \int_{\Omega_h} \partial_t \rho_h \left( \overline{ \gamma(m \bm M_h/\rho_h)-1 -\frac{\bm M_h^2}{\rho_h^2 \, c^2 \, \gamma(m \bm M_h/\rho_h)} } \right) c^2 \, dx + \int_{\Omega_h} \partial_t \bm M_h \cdot \bm w_h \, dx \\ \notag
&+ \frac{1}{4 \pi} \int_{\Omega_h} \left( \partial_t \bm E_h \cdot \bm E_h + \partial_t \bm B_h \cdot \bm B_h \right) \, dx  + \sum_k^N \frac{w_k \bm U_k}{m \gamma(\bm U_k)} \cdot \partial_t \bm U_k
\end{align}
To show that the energy rate is zero, we take $\phi_h = \left( \overline{ \gamma(m \bm M_h/ \rho_h) -1- \frac{\bm M \cdot \bm M}{\rho_h^2 c^2 \gamma(m \bm M_h/ \rho_h)} } \right) c^2$ in the density equation \eqref{eq:qqnrt_rho}, $\bm \mu_h= \bm w_h$ in the momentum equation \eqref{eq:qqnrt_M}, $\bm \nu_h=\bm E_h$ in  Ampere's equation \eqref{eq:qqnrt_E}, and $\bm \tau_h = \bm B_h$ in Faraday's equation \eqref{eq:qqnrt_B}, and arrive at
\begin{align}
\frac{d}{dt} \mathcal H = -\frac{e}{m} \sum_k^N \frac{w_k \, \bm U_k}{\gamma(\bm U_k)} \cdot \bm E_h(\bm X_k) +  \sum_k^N \frac{w_k \bm U_k}{m \gamma(\bm U_k)} \cdot \partial_t \bm U_k 
\end{align}
expressing the time rate of the particles' momentum with \eqref{eq:qqnrt_U_k}, we deduce the conservation of energy $\frac{d}{dt} \mathcal H=0$.

Next, we consider ´Ampere's equation \eqref{eq:qqnrt_E} and set the test function $\bm \nu_h \in \mathring N_k$ to $\nabla \phi_h$ with $\phi_h \in \mathring Q_{k+1}$. This is a legitimate choice, since according to \eqref{eq:prop_Qk}, $\nabla \phi_h \in \mathring N_k$.
\begin{align} \label{eq:prop_eq_e}
 & \int_{\Omega_h} \partial_t \bm E_h \cdot \nabla \phi_h \, dx = c \int_{\Omega_h} \bm B_h \cdot (\nabla \times \nabla \phi_h) \, dx - \frac{4 \pi e }{m} \int_{\Omega_h}   \rho_h \bm w_h \cdot \nabla \phi_h  \, dx \\ \notag
& \hspace{2.5cm} - \frac{4 \pi e}{m} \sum_k^{N} \frac{w_k \bm U_k}{\gamma(\bm U_k)}  \cdot \nabla \phi_h(\bm X_k)  \hspace{1cm} \forall \phi_h \in \mathring Q_{k+1}
\end{align}
On the left-hand side, we use the weak gradient definition \eqref{eq:weak_divergence}. On the right-hand side, we use the property $\nabla \times \nabla \phi_h=0$, see \eqref{eq:prop_Qk}, to get rid of the curl term. Then, we use the density equation \eqref{eq:qqnrt_rho} and express the second term in terms of the density rate. This is possible, since the density equation holds for all test functions in $Q_{k+1}$ and therefore for all the test functions in $\mathring Q_{k+1} \subset Q_{k+1}$ as well.
\begin{align} \label{eq:prop_eq_e}
 & \int_{\Omega_h} \partial_t (\nabla_w \cdot \bm E_h) \phi_h \, dx =  \frac{4 \pi e }{m} \int_{\Omega_h}  \partial_t \rho_h \phi_h  \, dx  + \frac{4 \pi e}{m} \sum_k^{N} \frac{w_k \bm U_k}{\gamma(\bm U_k)}  \cdot \nabla \phi_h(\bm X_k)  \hspace{1cm} \forall \phi_h \in \mathring Q_{k+1}
\end{align}
The particles' contribution on the right-hand side of \eqref{eq:prop_eq_e} represents the current due to the particles that we rewrite using Lemma \ref{prop:particle_mass_conservation}. Then the Ampere's equation \eqref{eq:prop_eq_e}, reads
\begin{align} \label{eq:prop_weak_gauss_law}
 & \int_{\Omega_h} \partial_t \left( \nabla_w \cdot \bm E_h - 4 \pi e \left(\frac{\rho_h}{m} + \sum_k^N w_k \delta(\bm x - \bm X_k)  \right) \right) \phi_h \, dx = 0\hspace{1cm} \forall \phi_h \in \mathring Q_{k+1}
\end{align}
and we see that the weak Gauss law \eqref{eq:prop_weak_gauss_law} is preserved.

Next, in the Faraday's equation \eqref{eq:qqnrt_B}, we take $\bm \tau_h = \nabla_w \varphi_h$ with $\varphi_h \in \mathring{DG}_k$. This is a legitimate choice, since $\nabla_w \varphi_h \in \mathring{RT}_k$ according to the weak gradient definition in \eqref{eq:weak_gradient}.
\begin{align} \label{eq:prop_divB}
\int_{\Omega_h} \partial_t \bm B_h \, \cdot \nabla_w \varphi_h \, dx = - \int_{\Omega_h} \nabla_w \varphi_h \cdot \nabla \times \bm E_h \hspace{1cm} \forall \varphi_h \in \mathring{DG}_k
\end{align}
By the property \eqref{eq:prop_Nk}, $\nabla \times \bm E_h \in \mathring {RT}_k$, and we can use the definition of the weak gradient \eqref{eq:weak_gradient} on the right-hand side to obtain $\int_{\Omega} \varphi_h \nabla \cdot \nabla \times \bm E_h \, dx$. This term vanishes due to property \eqref{eq:prop_Nk}. Applying the definition of the weak gradient to the left-hand side of \eqref{eq:prop_divB}, we deduce $ \int_{\Omega} \partial_t \nabla \cdot \bm B_h  \varphi_h \, dx =0$ for all $\varphi_h \in \mathring{DG}_k$.  Since $ \partial_t \nabla \cdot \bm B_h \in \mathring{DG}_k$ by the property \eqref{eq:prop_RTk}, we set $\varphi_h= \partial_t \nabla \cdot \bm B_h$ and deduce $\|\partial_t \nabla \cdot \bm B_h\|^2_2=0 \iff \partial_t \nabla \cdot \bm B_h = 0$. Which implies that $\nabla \cdot \bm B_h$ is preserved.
\end{proof}

\subsection{A formulation with fluxes}
Our second formulation is based on the choice $\rho_h \in DG_{k+1}, \bm M_h \in RT_k, \bm E_h \in \mathring N_k, \bm B_h \in \mathring{RT}_k$. Just like the discretization (\ref{eq:qqnrt_rho}-\ref{eq:qqnrt_U_k}), this discretization preserves total mass, total energy, weak Gauss law, and divB constraint. The main difference is that the density approximation belongs to $DG_{k+1}$ and is only piecewise continuous, whereas the momentum approximation belongs to $RT_k$ and has only its normal component continuous. Due to the discontinuities, it is necessary to consider fluxes over element boundaries. Since we are deriving our spatial discretization from the Poisson bracket, we have to take discontinuities into account in the bracket. For this reason, we consider integrals over each cell $K$ in the bracket and apply partial integration to the terms containing gradient and curl operators. Considering the discontinuities, we keep boundary terms which yields the following bracket:
\begin{align}\label{eq:bracket_with_boundary}
\{ \mathcal F, \mathcal G \}_K & = \int_K \left[ \frac{\delta \mathcal{G}}{ \delta \bm M} \cdot \nabla \frac{\delta \mathcal{F}}{\delta \bm M} -\frac{\delta \mathcal{F}}{ \delta \bm M} \cdot \nabla \frac{\delta \mathcal{G}}{\delta \bm M}\right] \cdot \bm M \, dx + \int_K \left[ \frac{\delta \mathcal{G}}{\delta \bm M} \cdot \nabla \frac{\delta \mathcal{F} }{\delta \rho} - \frac{\delta \mathcal{F}}{\delta \bm M} \cdot \nabla \frac{\delta \mathcal{G} }{\delta \rho}  \right] \, \rho \, dx\\ \label{eq:density_sub_bracket}
&+ \frac{4 \pi e}{m} \int_K \left[ \frac{\delta \mathcal{F}}{ \delta \bm M} \cdot \frac{\delta \mathcal{G} }{\delta \bm E}-\frac{\delta \mathcal{G} }{\delta \bm M} \cdot \frac{\delta \mathcal{F}}{\delta \bm E}   \right] \, \rho \, dx + \frac{e}{mc} \int_K \rho \bm B \cdot \left( \frac{\delta \mathcal{F}}{\delta \bm M} \times \frac{\delta \mathcal{G}}{\delta \bm M} \right) \, dx\\ \label{eq:coupling_sub_bracket}
&+ \sum_k^N \frac{1}{w_k} \left[ \frac{\partial \mathcal F}{\partial \bm X_k} \cdot \frac{\partial \mathcal G}{\partial \bm U_k} - \frac{\partial \mathcal G}{\partial \bm X_k} \cdot \frac{\partial \mathcal F}{\partial \bm U_k} \right] + 4 \pi \, e \sum_k^N \left. \left[ \frac{\partial \mathcal F}{\partial \bm U_k} \cdot \frac{\delta \mathcal G}{\delta \bm E} -  \frac{\partial \mathcal G}{\partial \bm U_k} \cdot \frac{\delta \mathcal F}{\delta \bm E} \right] \right\vert_{(\bm X_k, \bm U_k)}\\ \label{eq:maxwell_sub_bracket}
& + \frac{e}{c} \sum_k^N \frac{1}{w_k} \bm B(\bm X_k, t) \cdot \left( \frac{\partial \mathcal F}{\partial \bm U_k} \times \frac{\partial \mathcal G}{\partial \bm U_k} \right) 
+ 4 \pi c \int_K \left[ \frac{\delta \mathcal G}{\delta \bm B} \cdot \left( \nabla \times \frac{\delta \mathcal F}{\delta \bm E} \right)  -  \frac{\delta \mathcal F}{\delta \bm B} \cdot \left( \nabla \times \frac{\delta \mathcal G}{\delta \bm E} \right) \right] dx \\ \label{eq:m_boundary}
&  - \int_{\partial K} \left( \frac{\delta \mathcal G}{\delta \bm M} \cdot \bm n\right) \left( \frac{\delta \mathcal F}{\delta \bm M} \cdot \bm M \right)  \, ds + \int_{\partial_K} \left( \frac{\delta \mathcal F}{\delta \bm M} \cdot \bm n \right)\left( \frac{\delta \mathcal G}{\delta \bm M} \cdot \bm M \right)  \, ds  \\ \label{eq:density_boundary}
&  - \int_{\partial K}  \left( \rho  \frac{\delta \mathcal G}{\delta \bm M}  \right) \cdot \left( \frac{\delta \mathcal F}{\delta \rho } \bm n \right)  \, ds + \int_{\partial K}  \left( \rho \frac{\delta \mathcal G}{\delta \rho } \bm n \right) \cdot \left( \frac{\delta \mathcal F}{\delta \bm M} \right)  \, ds   \\ \label{eq:maxwell_boundary}
& - 4 \pi c \int_{\partial K} \left[ \frac{\delta \mathcal G}{\delta \bm B} \cdot \left( \bm n \times \frac{\delta F}{\delta \bm E} \right) \right] \, ds
+ 4 \pi c \int_{\partial K} \left[ \frac{\delta \mathcal F}{\delta \bm B} \cdot \left( \bm n \times \frac{\delta G}{\delta \bm E} \right) \right] \, ds
\end{align}
Lines (\ref{eq:m_boundary}-\ref{eq:maxwell_boundary}) contain the boundary terms that we added. We use the vector calculus identity $(\bm a \cdot \bm c)(\bm b \cdot \bm d)-(\bm b \cdot \bm c)(\bm a \cdot \bm d)=(\bm a \times \bm b) \cdot (\bm c \times \bm d)$ and rewrite the terms in line \eqref{eq:m_boundary} as follows
\begin{align}
\int_{\partial K} \left( \frac{\delta \mathcal F}{\delta \bm M}  \cdot \bm n \right) \left( \frac{\delta \mathcal G}{\delta \bm M} \cdot \bm M \right) \, ds -\int_{\partial K} \left( \frac{\delta \mathcal G}{\delta \bm M}  \cdot \bm n \right) \left( \frac{\delta \mathcal F}{\delta \bm M} \cdot \bm M \right) \, ds =  \int_{\partial K} \left( \bm n \times \bm M \right) \cdot \left( \frac{\delta \mathcal F}{\delta \bm M} \times \frac{\delta G}{\delta \bm M} \right) \, ds
\end{align}
Next, we assume time-independent scalar and vector fields $\phi(\bm x)$ and $\bm \mu(\bm x), \bm \nu(\bm x), \bm \tau(\bm x)$ that are smooth on element $K$. We consider functionals
\begin{align} \label{eq:functionals_for_weak_form_fluxes}
\int_{K} \rho \, \phi \, dx \hspace{0.5cm} \int_{K} \bm M \cdot \bm \mu \, dx \hspace{0.5cm} \int_{K} \bm E \cdot \bm \nu \, dx \hspace{0.5cm} \int_{K} \bm B \cdot \bm \tau \, dx 
\end{align}
In analogy to the derivation of the system (\ref{eq:weak_rho}-\ref{eq:U_k}), we use these functionals, the bracket (\ref{eq:bracket_with_boundary}-\ref{eq:maxwell_boundary}), and the Hamiltonian \eqref{eq:relativistic_hamiltonian_hybrid_degree_zero_conserved} to derive the following weak formulation of the equations of motion
\begin{align} 
\label{eq:rel_weak_form_rho_original}
& \int_{K} \partial_t \rho \, \phi = \int_{K} \rho \bm w \cdot \nabla \phi -  \int_{\partial K} \left( \rho \bm w \right) \cdot \left( \phi  \bm n \right) \\ \label{eq:rel_weak_form_M_original}
& \int_{K} \partial_t \bm M \cdot \bm \mu \, dx = \int_{K} \left( \bm w \cdot \nabla \bm \mu -\bm \mu \cdot \nabla \bm w \right) \cdot \bm M \, dx \\ \notag
& \hspace{2.5cm} - \int_{K} \rho \, \bm \mu \cdot \nabla \left( \gamma(m \bm M/ \rho) -1- \frac{\bm M^2}{\rho^2 c^2 \gamma(m \bm M/ \rho)} \right) c^2 \, dx \\ \notag
& \hspace{2.5cm}  + \frac{e}{m} \int_{K} \rho \left( \bm E + \frac{\bm M}{\rho \gamma(m \bm M/ \rho)} \times \frac{\bm B}{c} \right)  \cdot \bm \mu \, dx \\ \notag
&\hspace{2.5cm} + \int_{\partial K} \left( \bm n \times \bm M \right) \cdot \left( \bm \mu \times \bm w \right) \, ds\\ \notag
&\hspace{2.5cm} + \int_{\partial K} \left( \rho \left[ \gamma(m \bm M/ \rho) -1- \frac{\bm M^2}{\rho^2 c^2 \gamma(m \bm M/ \rho)}  \right] c^2 \right) \left( \bm \mu \cdot \bm n \right) ds \\ \label{eq:rel_weak_form_E_original}
& \int_{K} \partial_t \bm E \cdot \bm \nu \, dx = c \int_{K} \bm \nu \cdot (\nabla \times \bm B) \, dx - \frac{4 \pi e }{m} \int_{K}  \left( \rho \bm w \cdot \bm \nu \right) \, dx \\ \notag
& \hspace{2.5cm} - \frac{4 \pi e}{m} \sum_k^{N_K} \frac{w_k \bm U_k}{\gamma(\bm U_k)}  \cdot \bm \nu(\bm X_k) \, dx  -c \int_{\partial K} \bm B \cdot \left( \bm n \times \bm \nu \right) \, ds
\\ \label{eq:rel_weak_form_B_original}
& \int_{K} \partial_t \bm B \cdot \bm \tau \, dx = - c \int_{K} \bm E \cdot (\nabla \times \bm \tau ) \, dx + c \int_{\partial K} \bm \tau \cdot \left( \bm n \times \bm E \right) \, ds \\ \label{eq:rel_weak_form_w_original}
& \text{with } \bm w = \frac{\bm M}{\rho \gamma(m \bm M/\rho)}
\end{align}
And equations for the particles 
\begin{align} \label{eq:appendix_X_k}
&\partial_t \bm X_k(t) = \frac{\bm U_k}{m \gamma(\bm U_k)}\\ \label{eq:appendix_U_k}
&\partial_t \bm U_k(t) = e \left( \bm E(\bm X_k, t)+ \frac{\bm U_k \times \bm B(\bm X_k, t)}{cm \gamma(\bm U_k)} \right) 
\end{align}

Next, we replace the boundary terms over $\partial K$ with fluxes that enforce continuity weakly across edges $e$
\begin{align} \label{eq:weak_flux_1}
 &\sum_{K \in \mathcal{T}_h} \int_{\partial K} \left( \rho \bm w \right) \cdot \left( \phi  \bm n \right) \longrightarrow  \sum_{K \in \mathcal{T}_h} \int_{\partial K} \left( \rho \bm w \right)^{1*} \cdot \left( \phi  \bm n \right) =  \sum_{K \in \mathcal{T}_h} \int_{\partial K} \left( \rho \bm w \right)^{1*} \cdot \llbracket \phi \rrbracket \\ \notag
 &\sum_{K \in \mathcal{T}_h} \int_{\partial K} \left( \bm n \times \bm M \right) \cdot \left( \bm \mu \times \bm w  \right) \, ds  \longrightarrow  \sum_{K \in \mathcal{T}_h} \int_{\partial K} \left( \bm n \times \bm M \right)^{2*} \cdot \left( \bm \mu \times \bm w \right) = \\\label{eq:weak_flux_2}
& \hspace{6cm}  = \sum_{e \in \mathcal{E}_h} \int_{e} \left( \bm n \times (\bm M_h)^{2*} \right) \cdot \left \llbracket \bm \mu_h \times \bm w \right \rrbracket \, ds \\ \label{eq:weak_flux_3}
& \sum_{K \in \mathcal{T}_h} \int_{\partial K} \bm B \cdot (\bm n \times \bm \nu) ds \longrightarrow  \sum_{K \in \mathcal{T}_h} \int_{\partial K} \bm B \cdot (\bm n \times \bm \nu) ds  =  \sum_{e \in \mathcal{E}_h} \int_{\partial K} \bm B^{3*} \cdot \llbracket \bm n \times \bm \nu \rrbracket ds 
\end{align}
and 
\begin{align} \notag
&\sum_{K \in \mathcal{T}_h} \int_{\partial K} \rho\,  \bm \mu \cdot \bm n \left[  \gamma(m \bm M/ \rho) -1- \frac{\bm M^2}{\rho^2 c^2 \gamma(m \bm M/ \rho)} \right] c^2 \,  ds \\ \notag
& \hspace{2cm} \longrightarrow \sum_{K \in \mathcal{T}_h} \int_{\partial K} \rho\,  \bm \mu \cdot \bm n \left[  \gamma(m \bm M/ \rho) -1- \frac{\bm M^2}{\rho^2 c^2 \gamma(m \bm M/ \rho)} \right] c^2 \,  ds =  \\ \label{eq:weak_flux_4}
& \hspace{2cm} = \sum_{e \in \mathcal{E}_h} \int_{e}  \left(\rho \bm \mu \right)^{1*} \cdot \left \llbracket  \gamma(m \bm M/ \rho) -1- \frac{\bm M^2}{\rho^2 c^2 \gamma(m \bm M/ \rho)}  \right \rrbracket c^2 \, ds \\ \label{eq:weak_flux_5}
&\sum_{K \in \mathcal{T}_h} \int_{\partial K} \bm \tau \cdot \left( \bm n \times \bm E \right) \implies \sum_{e \in \mathcal{E}_h } \int_e (\bm \tau)^{3*} \cdot \left \llbracket \bm n \times \bm E \right \rrbracket 
\end{align}
Since we would consider $\bm \nu_h \in \mathring{N}_k$ and $\bm E_h \in \mathring{N}_k$, we omit the terms that vanish for the upcoming choice of function spaces: $\sum_{e \in \mathcal{E}_h} \int_{\partial K} \bm B^{3*} \cdot \llbracket \bm n \times \bm \nu \rrbracket ds$ and $\sum_{e \in \mathcal{E}_h } \int_e (\bm \tau)^{3*} \cdot \left \llbracket \bm n \times \bm E \right \rrbracket $. For the remaining fluxes, we use the following upwind flux denoted by $(\cdot)^*$ 
\begin{align}
(\cdot)^* = \{\cdot\} + \frac{1}{2} \frac{\bm M \cdot \bm n}{|\bm M \cdot \bm n|} \llbracket \cdot \rrbracket
\end{align}

The element local weak formulation (\ref{eq:rel_weak_form_rho_original}-\ref{eq:particle_U_k}) then leads to global weak formulation with weak continuity between the elements according to (\ref{eq:weak_flux_1}-\ref{eq:weak_flux_5}).

Inserting our discrete fields, this yields the following semi-discrete system of equations.
\begin{align} 
\label{eq:dgdiv_rho}
& \int_{\Omega_h} \partial_t \rho_h \, \phi_h = \int_{\Omega_h} \rho_h \bm w_h \cdot \nabla \phi_h -  \sum_{e \in \mathcal E_h} \int_{e} \left( \rho_h \bm w_h \right)^{*} \cdot \llbracket \phi_h  \rrbracket \hspace{1cm} \forall \phi_h \in DG_{k+1}  \\ \label{eq:dgdiv_m}
& \int_{\Omega_h} \partial_t \bm M_h \cdot \bm \mu_h \, dx = \int_{\Omega_h} \left( \bm w_h \cdot \nabla \bm \mu_h -\bm \mu_h \cdot \nabla \bm w_h \right) \cdot \bm M_h \, dx \\ \notag
& \hspace{2.5cm} - \int_{\Omega_h} \rho_h \, \bm \mu_h \cdot \nabla \left( \overline{ \gamma(m \bm M_h/ \rho_h) -1- \frac{\bm M_h^2}{c^2 \, \rho_h^2  \gamma(m \bm M_h/ \rho_h)} } \right) c^2 \, dx \\ \notag
& \hspace{2.5cm}  + \frac{e}{m} \int_{\Omega_h} \rho_h \left( \bm E_h + \bm w_h \times \frac{\bm B_h}{c} \right)  \cdot \bm \mu_h \, dx \\ \notag
&\hspace{2.5cm} + \sum_{e \in \mathcal E_h} \int_{e } \left( \bm n \times \bm M_h \right)^{*} \cdot \llbracket \bm \mu_h \times \bm w_h \rrbracket \, ds\\ \notag
&\hspace{2.5cm} + \sum_{e \in \mathcal{E}_h} \int_{e}  \left(\rho_h \bm \mu_h \right)^{*} \cdot \left \llbracket \overline{ \gamma(m \bm M_h/ \rho_h) -1- \frac{\bm M_h^2}{c^2 \, \rho_h^2 \gamma(m \bm M_h/ \rho_h)} } \right \rrbracket c^2 \, ds \hspace{0.5cm}\forall \bm \mu_h \in RT_k\\ \label{eq:dgdiv_E}
& \int_{\Omega_h} \partial_t \bm E_h \cdot \bm \nu_h \, dx = c \int_{\Omega_h} \bm B_h \cdot (\nabla \times \bm \nu_h) \, dx - \frac{4 \pi e }{m} \int_{\Omega_h}  \left( \rho_h \bm w_h \cdot \bm \nu_h \right) \, dx \\ \notag
& \hspace{7cm} - \frac{4 \pi e}{m} \sum_k^{N_K} \frac{w_k \bm U_k}{\gamma(\bm U_k)}  \cdot \bm \nu_h(\bm X_k)  \hspace{1cm} \bm \nu_h \in \mathring{N}_k
\\ \label{eq:dgdiv_B}
& \int_{\Omega_h} \partial_t \bm B_h \cdot \bm \tau_h \, dx = - c \int_{\Omega_h} \bm \tau_h \cdot (\nabla \times \bm E_h) \, dx \hspace{1cm} \forall \bm \tau_h \in \mathring{RT}_k \\ \label{eq:dgdiv_w}
& \text{with } \bm w_h = \overline{\frac{\bm M_h}{\rho_h \gamma(m \bm M_h/\rho_h)}}
\end{align}
\begin{align} \label{eq:dgdiv_X_k}
&\partial_t \bm X_k(t) = \frac{\bm U_k}{m \gamma(\bm U_k)} \hspace{0.5cm} \text{for } \bm X_k \in \Omega_h\\ \label{eq:dgdiv_U_k}
&\partial_t \bm U_k(t) = e \left( \bm E_h(\bm X_k, t)+ \frac{\bm U_k \times \bm B_h(\bm X_k, t)}{cm \gamma(\bm U_k)} \right) 
\end{align}
Here, the overbar in the velocity \eqref{eq:dgdiv_w} denotes the $L^2(\Omega)$-projection onto $RT_k$, and the overbars in the momentum equation \eqref{eq:dgdiv_m} denote the $L^2(\Omega)$-projection onto $DG_{k+1}$. 

\begin{proposition} \label{prop:dgdiv_proof}
The scheme (\ref{eq:dgdiv_rho}-\ref{eq:dgdiv_U_k}), conserves total fluid mass $\int_{\Omega_h} \rho \, dx$, total particle mass $\sum_k^N w_k$, total energy $\int_{\Omega_h} \rho_h (\gamma(m \bm M_h/ \rho_h)-1) c^2 dx + \sum_k^N w_k (\gamma(\bm U_k)-1)m c^2 +\frac{1}{8 \pi} \int_{\Omega_h} \left( \bm E_h^2 + \bm B_h^2 \right) dx$, weak Gauss law $ \int_{\Omega_h} \left[ \nabla_w \cdot \bm E_h - 4 \pi e \left(\frac{\rho_h}{m} + \sum_k^N w_k \delta(\bm x - \bm X_k) \right) \right] \, \phi_h \, dx$ for all $\phi_h \in \mathring Q_{k+1}$, and $\nabla \cdot \bm B_h$ constraint.
\end{proposition}

\begin{proof}
The proof proceeds as in Proposition \ref{prop:h1h1_proof} but also considering the flux terms. The total mass is still preserved, since the flux term in the density equation \eqref{eq:dgdiv_rho} vanishes for $\phi_h=1$.
Energy conservation is not affected, since in the momentum equation \eqref{eq:dgdiv_m} the flux involving $\llbracket \bm \mu_h \times \bm w_h \rrbracket$ vanishes for $\bm \mu_h = \bm w_h$. The flux in the density equation \eqref{eq:dgdiv_rho} cancels with the second flux in the momentum equation \eqref{eq:dgdiv_m} when choosing the test function $\phi_h = \left( \overline{ \gamma(m \bm M_h/ \rho_h) -1- \frac{\bm M_h^2}{c^2 \, \rho_h^2  \gamma(m \bm M_h/ \rho_h)} } \right) c^2 $ and $\bm \mu_h=\bm w_h$. The weak Gauss law holds, since $\mathring Q_{k+1} \subset DG_{k+1}$ and we may take $\phi_h \in \mathring Q_{k+1}$. The divB constraint is not affected.
\end{proof}

\section{Temporal discretization} \label{sec:time_integrators}

In this section, we discuss two temporal discretizations for the systems (\ref{eq:qqnrt_rho}-\ref{eq:qqnrt_U_k}) and (\ref{eq:dgdiv_rho}-\ref{eq:dgdiv_U_k}). First, we propose an implicit scheme derived from the average vector field method that preserves all the original invariants of the semi-discrete methods. Then, we look at the invariants of a high-order explicit strong-stability-preserving Runge--Kutta scheme. 

\subsection{An energy-conserving implicit scheme}\label{sec:avf}

\vspace{2.5mm}
\noindent \textbf{Average Vector Field (AVF) gradient.}

Let us first review the average vector field (AVF) method \cite{average_vector} that we use for our time-discretization. For this we consider a vector $\bm u^{k}$ of the discrete dynamic variables at time step $k$ and the discrete Hamiltonian function $H(\bm u^{k})$. Moreover, we define a continuous in time solution based on a linear interpolation between two consecuitive time steps and denote it, for $\alpha \in [0,1]$, by
\begin{align} \label{eq:alpha_interpolation}
\bm u^{k,\alpha} = (1-\alpha) \, \bm u^{k+1} + \alpha \, \bm u^{k}.
\end{align}

The main property of the AVF gradient is that for a function $H(\bm u)$, depending on the vector of degrees of freedom $\bm u$, the difference  $H(\bm u^{k+1})-H({\bm u}^{k})$ can be expressed in terms of the derivative of $H(\bm u)$ and the difference of the degrees of freedom, i.e.
\begin{align} \notag
\frac{H(\bm u^{k+1})-H({\bm u}^{k})}{\Delta t} &= - \frac{1}{\Delta t} \int_0^1 \frac{dH \left((1-\xi) {\bm u}^{k+1}+\xi {\bm u}^k \right)}{d \xi} \, d \xi \\ \notag
&= \int_0^1 \frac{\partial H}{\partial {\bm u}}\left( (1-\xi) {\bm u}^{k+1}+\xi  {\bm u}^k \right) \cdot \frac{\left( {\bm u}^{k+1}-{\bm u}^{k}  \right)}{\Delta t} \, d \xi \\ \label{eq:avf_energy}
&= \int_0^1 \frac{\partial H}{\partial {\bm u}}\left( {\bm u}^{\xi} \right) \, d \xi \cdot \frac{\left( {\bm u}^{k+1}-{\bm u}^{k}\right)}{\Delta t}
\end{align}
here the integral $\int_0^1 \frac{\partial H}{\partial {\bm u}}\left( {\bm u}^{\xi} \right) \, d \xi$ performs some averaging of the derivative.

A system of ordinary differential equations in the form $\partial_t \bm u = \bm J(\bm u) \cdot \frac{\partial H}{\partial \bm u}$ is discretized with the AVF gradient as follows
\begin{align} \label{eq:avf_discretization}
\frac{\bm u^{k+1}-\bm u^k}{\Delta t} = \bm J(\bm u^{k},\bm u^{k+1}) \cdot \int_{0}^{1} \frac{\partial H}{\partial \bm u} (\bm u^{\xi}) \, d \xi
\end{align}

If $\bm J(\bm u^{k},\bm u^{k+1})$ is antisymmetric, i.e., $\bm J^\top = - \bm J$, then energy is conserved. This is checked by plugging \eqref{eq:avf_discretization} in \eqref{eq:avf_energy} and using the antisymmetry property of $\bm J$.

The AVF gradient can also be applied to functionals. Let $\mathcal H[\bm u_h]$ be a functional of the function $\bm u_h$. The functional $\mathcal H[\bm u_h]$ can be viewed as a function $H({\bm u})$ of the degrees of freedom $\bm u$. By comparing Fr\'echet's derivatives of  $\mathcal H[\bm u_h]$ and $H({\bm u})$ one can deduce
$\frac{\partial H}{\partial {u_i}}({\bm u}) = \int_{\Omega} \frac{\delta \mathcal H}{\delta \bm u_h}(\bm u_h) \cdot \bm \mu_i \, dx$.
This leads to the AVF gradient for the functional $\mathcal H[\bm u_h]$
\begin{align} \label{eq:avf_method}
\frac{\mathcal H[\bm u_h^{k+1}]-\mathcal H[\bm u_h^{k}]}{\Delta t} =   \int_0^1 \frac{\partial H}{\partial {\bm u}}\left( {\bm u}^{\xi} \right) \, d \xi \cdot \frac{\left( {\bm u}^{k+1}-{\bm u}^{k}\right)}{\Delta t} 
=  \int_{\Omega_h} \left[ \left(  \int_0^1   \frac{\delta \mathcal H}{\delta \bm u_h}(\bm u^{\xi}_h)\, d \xi  \, \right) \cdot  \frac{\bm u_h^{k+1}-\bm u_h^{k}}{\Delta t}  \, \right] dx  
\end{align}

\vspace{2.5mm}
\noindent \textbf{Implicit temporal discretization.}\\ 
We now discuss how we choose the discrete Poisson matrix $\bm J( \bm u^{k}, \bm u^{k+1})$. Since we aim for a scheme that preserves Gauss' law at the discrete level, we need a discrete analog of Lemma \ref{prop:particle_mass_conservation}.
We consider the particle positions at the old and new time steps: $\bm X_p^{k}$ and $\bm X_p^{k+1}$. We assume that within each time-step, each particle has a linear trajectory that crosses $s$ cells, where $s$ is an integer. The trajectory is partitioned into $s$ sub-intervals that we denote by $[\bm A_{p,i-1}, \bm A_{p, i}]$, $i=0,\ldots, s$, with $\bm A_{p,0} = \bm X_p^k$, $\bm A_{p,s} = \bm X_p^{k+1}$ and $\bm A_{p,i}$, $i=1, \ldots, s-1$ being the intersection points of the line between $\bm X_p^k$ and $\bm X_p^{k+1}$ and each of the element boundaries crossed by the line. We also define $\bm X_{p,i}^{\xi}:=(1-\xi) \bm A_{p,i} + \xi \bm A_{p,i-1}$.

\begin{lemma} \label{prop:pic_mass_discrete}
Discrete weak conservation of charge for particles reads
\begin{align} \notag
\int_{\Omega_h} \sum_p^{N} w_p  \frac{\left( \delta(\bm x-\bm X^{k+1}_p) -\delta(\bm x-\bm X^{k}_p) \right)}{\Delta t} \phi_h(\bm x) - \sum_p^{N} \, w_p \sum_{i}^s\, \frac{\bm A_{p,i}-\bm A_{p,i-1}}{\Delta t} \cdot \int_0^1 \nabla \phi_h(\bm X^{\xi}_{p,i}) \, d \xi= 0 \\ \hspace{1cm} \forall \phi_h(\bm x) \in \mathring Q_{k+1} 
\end{align}
\end{lemma}
\begin{proof}
\begin{align} \notag
\int_{\Omega_h} \sum_p^{N} w_p  \frac{\left( \delta(\bm x-\bm X^{k+1}_p) -\delta(\bm x-\bm X^{k}_p) \right)}{\Delta t} & \phi_h(\bm x) \stackrel{\text{(sifting property)}}{=}  \sum_p^{N} w_p \frac{\phi_h(\bm X_p^{k+1}) - \phi_h(\bm X_p^k)}{\Delta t} \\ \notag
&= \sum_p^{N} \sum_i^s w_p \frac{\phi_h(\bm A_{p,i}) - \phi_h(\bm A_{p,i-1})}{\Delta t} \\ \label{eq:discrete_current}
& \stackrel{\text{(AVF)}}{=}   \sum_p^{N} \sum_i^s w_p \frac{\bm A_{p,i}-\bm A_{p,i-1}}{\Delta t} \cdot \int_{0}^1 \nabla \phi_h(\bm X^{\xi}_{p,i}) d \xi
\end{align}
\end{proof}
In analogy to Lemma \ref{prop:particle_mass_conservation}, we recognize that Eq. \eqref{eq:discrete_current} states the divergence of the discrete current due to particles.

We use the piecewise integral to define the evaluation of the current in our AVF scheme. To keep the necessary antisymmetry in the Poisson matrix, we use the same definition in the evaluation of the electric field in the update of $\bm U$. For the density $\rho$, we allow for a combination of the old and new time step as in \eqref{eq:alpha_interpolation} by a parameter $\theta \in [0,1]$.

For the semi-discrete system (\ref{eq:qqnrt_rho}-\ref{eq:qqnrt_U_k}), this yields the following fully discrete scheme
\begin{align} 
\label{eq:dt_qqnrt_rho}
& \int_{\Omega_h} \frac{\rho^{k+1}-\rho^{k}}{\Delta t} \, \phi_h = \int_{\Omega_h} \rho_h^{k, \theta} \bm w_h \cdot \nabla \phi_h  \hspace{1cm} \forall \phi_h \in Q_{k+1}  \\ \label{eq:dt_qqnrt_m}
& \int_{\Omega_h} \frac{\bm M^{k+1}-\bm M^k}{\Delta t} \cdot \bm \mu_h \, dx = \int_{\Omega_h} \left( \bm w_h \cdot \nabla \bm \mu_h -\bm \mu_h \cdot \nabla \bm w_h \right) \cdot \bm M_h^{k+1/2} \, dx \\ \notag
& \hspace{1cm} - \int_{\Omega_h} \rho_h^{k,\theta} \, \bm \mu_h \cdot \nabla \left( \overline{ \int_{0}^1 \gamma(m \bm M_h^{\xi}/ \rho_h^{\xi}) -1- \frac{ \left( \bm M_h^{\xi} \right) ^2}{c^2 \, \left(\rho_h^{\xi} \right)^2  \gamma(m \bm M_h^{\xi}/ \rho_h^{\xi})} \, d \xi} \right) c^2 \, dx \\ \notag
& \hspace{1cm}  + \frac{e}{m} \int_{\Omega_h} \rho_h^{k,\theta} \left( \bm E_h^{k+1/2} + \bm w_h \times \frac{\bm B_h^{k+1/2}}{c} \right)  \cdot \bm \mu_h \, dx 
\hspace{1cm} \forall \bm \mu_h \in Q_{k+1}^d\\ \label{eq:dt_qqnrt_E}
& \int_{\Omega_h} \frac{\bm E_h^{k+1} - \bm E_h^k}{\Delta t} \cdot \bm \nu_h \, dx = c \int_{\Omega_h} \bm B_h^{k+1/2} \cdot (\nabla \times \bm \nu_h) \, dx - \frac{4 \pi e }{m} \int_{\Omega_h}  \left( \rho_h^{k,\theta} \bm w_h \cdot \bm \nu_h \right) \, dx \\ \notag
& \hspace{5cm} - 4 \pi e  \sum_p^N \sum_i^s w_p  \frac{( \bm A_{p,i}- \bm A_{p,i-1} ) }{\Delta t} \cdot \int_{0}^1 \bm  \nu_h(\bm X^{\xi}_{p,i}) d \xi   \hspace{1cm} \bm \nu_h \in \mathring{N}_k
\\ \label{eq:dt_qqnrt_B}
& \int_{\Omega_h} \frac{\bm B_h^{k+1}- \bm B_h^k}{\Delta t} \cdot \bm \tau_h \, dx = - c \int_{\Omega_h} \bm \tau_h \cdot (\nabla \times \bm E_h^{k+1/2}) \, dx \hspace{1cm} \forall \bm \tau_h \in \mathring{RT}_k \\ \label{eq:dt_qqnrt_w}
& \text{with } \bm w_h = \overline{ \int_{0}^1 \frac{\bm M_h^{\xi}}{\rho_h^{\xi} \gamma(m \bm M_h^{\xi}/\rho_h^{\xi})} d \xi}
\end{align}
\begin{align} \label{eq:dt_qqnrt_X_k}
&\frac{\bm X_p^{k+1}- \bm X_p^k}{\Delta t}= \int_0^1 \frac{\bm U_p^{\xi}}{m \gamma(\bm U_p^{\xi})} d \xi \hspace{0.5cm} \text{for } \bm X_k \in \Omega_h \\ \label{eq:dt_qqnrt_U_k}
&\frac{\bm U_p^{k+1}-\bm U_p^k}{\Delta t} = e \left( \sum_i^s \bm D_i \cdot \int_{0}^1 \bm E_h^{k+1/2}(\bm X^{\xi}_{p,i}) d \xi + \frac{1}{m} \int_0^1 \frac{\bm U_p^{\xi}}{\gamma(\bm U_p^{\xi})} d \xi \times \frac{\bm B_h^{k+1/2}(\bm X_p^{k+1/2}, t)}{c} \right) 
\end{align}
In our numerical experiments, we set $\theta=\frac{1}{2}$.

We also introduced diagonal matrices in the particle momentum equation \eqref{eq:dt_qqnrt_U_k}
\begin{align}
\bm D_i := \text{diag}\left(
\frac{ (\bm A_{p,i} - \bm A_{p,i-1})_1}{ \left(\bm X_p^{k+1} - \bm X_p^{k} \right)_1},\frac{(\bm A_{p,i} - \bm A_{p,i-1})_2  }{\left(\bm X_p^{k+1} - \bm X_p^{k} \right)_2}, \frac{(\bm A_{p,i} - \bm A_{p,i-1})_3 }{\left(\bm X_p^{k+1} - \bm X_p^{k} \right)_3} \right)
\end{align}
these matrices $\bm D_i$ have the following two properties 
\begin{align} \label{eq:D_prop_1}
\sum_i^s \bm D_i = \bm I \text{ where } \bm I \text{ is the identity matrix} \\ \label{eq:D_prop_2}
\bm D_i \cdot \int_0^1 \frac{\bm U_p^\xi}{m \gamma(\bm U_p^\xi)} \, d \xi = \left( \bm A_{p,i} - \bm A_{p,i-1} \right) / \Delta t
\end{align}
property \eqref{eq:D_prop_1} indicates that the right-hand side term $\bm D_i \cdot \int_{0}^1 \bm E^{k+1/2}(\bm X^{k+\xi}_{p,i}) d \xi$ in the particle momentum equation \eqref{eq:dt_qqnrt_U_k} performs averaging of the electric field components across different cells in proportion to the paths' lengths. Property \eqref{eq:D_prop_2} holds due to the particle position equation \eqref{eq:dt_qqnrt_X_k} and is essential for the energy conservation. The matrix $\bm D_i$ is not well defined in case $\bm X_p^{k+1}-\bm X_p^{k}=0$. In this case $s=1$, i.e. the trajectory is represented by a point. We can view this case as $\lim_{\bm X_p^{k+1} \rightarrow \bm X_p^k} \bm D_1=\bm I$.

\begin{proposition} \label{prop:dt_qqnrt}
The temporal discretization in (\ref{eq:dt_qqnrt_rho}-\ref{eq:dt_qqnrt_U_k}) exactly conserves total fluid mass $\int_{\Omega_h} \rho_h^k \, dx$, total particle mass $\sum_p^N w_p$, total energy $\int_{\Omega_h} \rho_h^k (\gamma(m \bm M_h^k/ \rho_h^k)-1) c^2 dx + \sum_p^N w_p (\gamma(\bm U_p^k)-1)m c^2$ $+\frac{1}{8 \pi} \int_{\Omega_h} \left( (\bm E_h^k)^2 + \right.$ $\left. (\bm B_h^k)^2 \right) dx$, the weak Gauss's law  $ \int_{\Omega_h} \left[ \nabla_w \cdot \bm E^k -  4 \pi e \left(\frac{\rho^k_h}{m} + \sum_p^N w_p \delta(\bm x - \bm X_p^k) \right) \right] \, \phi_h \, dx$ for all $\phi_h \in \mathring Q_{k+1}$, and $\nabla \cdot \bm B_h^k$ constraint. 
\end{proposition}
\begin{proof}
We take $\phi_h = 1$ in \eqref{eq:dt_qqnrt_rho} to deduce conservation of total mass
\begin{align}
\int_{\Omega_h} \rho_h^{k+1} \, dx = \int_{\Omega_h} \rho_h^k \, dx
\end{align}
To check conservation of total energy, we consider discrete energy rate
\begin{align} \label{eq:diff_row_1}
\frac{\mathcal H^{k+1}-\mathcal H^{k}}{\Delta t} &= \frac{1}{\Delta t} \int_{\Omega_h} \rho_h^{k+1} \left( \gamma \left(m \bm M_h^{k+1}/\rho_h^{k+1} \right) -1 \right) c^2 \, dx - \frac{1}{\Delta t} \int_{\Omega_h} \rho_h^{k} \left( \gamma \left( m \bm M_h^k/ \rho_h^k \right) -1 \right) c^2 \, dx\\ \label{eq:diff_row_2}
&+ \frac{1}{\Delta t}\sum_p^N w_p (\gamma(\bm U_p^{k+1})-1)m c^2-\frac{1}{\Delta t}\sum_p^N w_p (\gamma(\bm U_p^{k})-1)m c^2 \\ \label{eq:diff_row_3}
&+\frac{1}{\Delta t}\frac{1}{8 \pi} \int_{\Omega_h} \left( (\bm E_h^{k+1})^2 + (\bm B_h^{k+1})^2 \right) \, dx -\frac{1}{\Delta t}\frac{1}{8 \pi} \int_{\Omega_h} \left( (\bm E_h^{k})^2 + (\bm B_h^{k})^2 \right) \, dx 
\end{align}
applying the AVF gradient \eqref{eq:avf_method} to \eqref{eq:diff_row_1}, the AVF gradient \eqref{eq:avf_energy} to \eqref{eq:diff_row_2} and expressing the functional and partial derivatives, appearing after the application of the AVF gradients, with the help of \eqref{eq:functional_der_rel_hamil} we obtain
\begin{align} \notag
\frac{\mathcal H^{k+1}-\mathcal H^{k}}{\Delta t} &= \int_{\Omega_h} \frac{\rho_h^{k+1}-\rho_h^{k}}{\Delta t}  \cdot \left(  \int_0^1  \left(\gamma(m \bm M_h^\xi/\rho_h^\xi) - 1-\frac{(\bm M_h^{\xi})^2}{\gamma(m \bm M_h^\xi/\rho_h^\xi) \, (\rho_h^{\xi})^2 \, c^2} \right) c^2 \, d \xi  \, \right) dx \\ \notag
&+\int_{\Omega_h} \frac{\bm M_h^{k+1}-\bm M_h^{k}}{\Delta t}  \cdot \left(  \int_0^1  \frac{\bm M_h^{\xi}}{\gamma(m \bm M_h^\xi/\rho_h^\xi) \, \rho^{k+\xi}}\, d \xi  \, \right) dx \\ \notag
 &+ \sum_p^N  w_p  \frac{\bm U^{k+1}_p-\bm U^{k}_p}{\Delta t} \cdot \int_0^1 \frac{\bm U_p^{k+\xi}}{m \gamma(\bm U_p^{k+\xi})} \, d \xi \\ \notag
 & +\frac{1}{8 \pi} \int_{\Omega_h} \left( \frac{(\bm E_h^{k+1})^2-(\bm E_h^{k})^2}{\Delta t} + \frac{(\bm B_h^{k+1})^2 -(\bm B_h^{k})^2}{\Delta t}  \right) \, dx 
\end{align}
Next, we use the definition of the $L^2$-projection for the first and the second terms on the right-hand side
\begin{align} \label{eq:proof_discrete_hamiltonian_diffference}
\frac{\mathcal H^{k+1}-\mathcal H^{k}}{\Delta t} &= \int_{\Omega_h} \frac{\rho_h^{k+1}-\rho_h^{k}}{\Delta t}  \cdot \left(  \overline{\int_0^1  \left(\gamma(m \bm M_h^\xi/\rho_h^\xi) - 1-\frac{(\bm M_h^{\xi})^2}{\gamma(m \bm M_h^\xi/\rho_h^\xi) \, (\rho^{k+\xi})^2 \, c^2} \right) c^2 \, d \xi }  \, \right) dx \\ \notag
&+\int_{\Omega_h} \frac{\bm M_h^{k+1}-\bm M_h^{k}}{\Delta t}  \cdot \left( \overline{  \int_0^1  \frac{\bm M_h^{\xi}}{\gamma(m \bm M_h^\xi/\rho_h^\xi) \, \rho^{k+\xi}}\, d \xi  } \, \right) \,  dx \\ \notag
 &+ \sum_p^N  w_p  \frac{\bm U^{k+1}_p-\bm U^{k}_p}{\Delta t} \cdot \int_0^1 \frac{\bm U_p^{k+\xi}}{m \gamma(\bm U_p^{k+\xi})} \, d \xi \\ \notag
 & +\frac{1}{4 \pi} \int_{\Omega_h} \left(\frac{ \bm E_h^{k+1}-\bm E_h^{k}}{\Delta t} \cdot \bm E_h^{k+1/2}  + \frac{ \bm B_h^{k+1}-\bm B_h^{k}}{\Delta t} \cdot \bm B_h^{k+1/2}  \right) \, dx 
\end{align}
then we consider \eqref{eq:dt_qqnrt_rho} with $\phi_h =\left(  \overline{\int_0^1  \left(\gamma(m \bm M_h^\xi/\rho_h^\xi) - 1-\frac{(\bm M_h^{\xi})^2}{\gamma(m \bm M_h^\xi/ \rho_h^\xi) \, (\rho_h^{\xi})^2 \, c^2} \right) c^2 \, d \xi }  \, \right)$, \eqref{eq:dt_qqnrt_m} with $\bm \mu_h = \bm w_h$, \eqref{eq:dt_qqnrt_E} with $\bm \nu_h=\bm E^{k+1/2}$, and \eqref{eq:dt_qqnrt_B} with $\bm \tau_h = \bm B^{k+1/2}$ and use them in place of the terms on the right-hand side of \eqref{eq:proof_discrete_hamiltonian_diffference}. Considering also the evolution equation for the particle's momentum \eqref{eq:dt_qqnrt_U_k}, along with the second property of the $\bm D_i$ matrix \eqref{eq:D_prop_2}, it can be shown that all the terms cancel out. So the total energy is conserved.

To check divB constraint, we consider the Faraday's law \eqref{eq:dt_qqnrt_B} and take $\bm \tau_h=\nabla_w \varphi_h$ with $\varphi_h \in \mathring{DG}_k$. 
\begin{align} \label{eq:eq_b_euler}
\int_{\Omega_h} \frac{\bm B_h^{k+1} - \bm B_h^k}{\Delta t} \cdot \nabla_w \varphi_h \, dx = -c \int_{\Omega_h}\nabla_w \varphi_h\cdot \left( \nabla \times \bm E_h^{k+1/2} \right) \, dx \hspace{1cm} \forall \varphi_h \in \mathring{DG}_k
\end{align}
the right-hand side of \eqref{eq:eq_b_euler} vanishes due to the property \eqref{eq:prop_Nk} that states $\nabla \cdot \nabla \times \bm E_h^k=0$. On the left-hand side of \eqref{eq:eq_b_euler}, we take $\varphi_h = \nabla \cdot (\bm B_h^{k+1}-\bm B_h^k) \in \mathring{DG}_k$ and deduce $\nabla \cdot \bm B_h^{k+1} = \nabla \cdot \bm B_h^k$.

To check the Gauss's law, we consider the Ampere's equation \eqref{eq:dt_qqnrt_E} and take $\bm \nu_h=\nabla \phi_h$ with $\phi_h \in \mathring Q_{k+1}$
\begin{align} \label{eq:eq_e}
& \int_{\Omega_h} \frac{\bm E_h^{k+1} - \bm E_h^k}{\Delta t} \cdot \nabla \phi_h \, dx = c \int_{\Omega_h} \bm B_h^{k+1/2} \cdot (\nabla \times \nabla \phi_h) \, dx - \frac{4 \pi e }{m} \int_{\Omega_h}  \left( \rho_h^{k,\theta} \bm w_h \cdot \nabla \phi_h \right) \, dx \\ \notag
& \hspace{5cm} - 4 \pi e  \sum_p^N \sum_i^s w_p  \frac{( \bm A_{p,i}- \bm A_{p,i-1} ) }{\Delta t} \cdot \int_{0}^1 \nabla \phi_h(\bm X^{\xi}_{p,i}) d \xi \hspace{1cm} \phi_h \in \mathring Q_{k+1}
\end{align}
the term with $\nabla \times \nabla \phi_h$ vanishes due to the property \eqref{eq:prop_Qk}. We use the weak divergence definition \eqref{eq:weak_divergence} to rewrite the left-hand side of \eqref{eq:eq_e}. We then use the density equation \eqref{eq:dt_qqnrt_rho} and Lemma \ref{prop:pic_mass_discrete} to rewrite the right-hand side:
\begin{align} \notag
\int_{\Omega_h} \left( \nabla_w \cdot \bm E_h^{k+1}-\nabla_w \cdot \bm E_h^{k} \right) &= \frac{4 \pi e}{m} \int_{\Omega_h} \left( \rho_h^{k+1}-\rho_h^{k} \right) \phi_h \, dx \\ 
& + 4 \pi e \int_{\Omega_h} \sum_p^N w_p \left( \delta \left(\bm x - \bm X_p^{k+1} \right) - \delta \left(\bm x - \bm X_p^{k} \right) \right) \phi_h \, dx \\ \notag
& \forall \phi_h \in \mathring Q_{k+1}
\end{align}
and we see that the Gauss's law is preserved.
\end{proof}

The temporal discretization for the semi-discrete scheme with fluxes (\ref{eq:dgdiv_rho}-\ref{eq:dgdiv_U_k}) is similar to (\ref{eq:dt_qqnrt_rho}-\ref{eq:dt_qqnrt_U_k}). It has the same conservation properties. For the benefit of the reader we summarize it below.
\begin{align} 
\label{eq:dt_dgdiv_rho}
& \int_{\Omega_h} \frac{\rho^{k+1}-\rho^{k}}{\Delta t} \, \phi_h = \int_{\Omega_h} \rho_h^{k,\theta} \bm w_h \cdot \nabla \phi_h -  \sum_{e \in \mathcal E_h} \int_{e} \left( \rho_h^{k,\theta_2} \bm w_h \right)^{*} \cdot \llbracket \phi_h  \rrbracket \hspace{1cm} \forall \phi_h \in DG_{k+1}  \\ \label{eq:dt_dgdiv_m}
& \int_{\Omega_h} \frac{\bm M^{k+1}-\bm M^k}{\Delta t} \cdot \bm \mu_h \, dx = \int_{\Omega_h} \left( \bm w_h \cdot \nabla \bm \mu_h -\bm \mu_h \cdot \nabla \bm w_h \right) \cdot \bm M_h^{k+1/2} \, dx \\ \notag
& \hspace{1cm} - \int_{\Omega_h} \rho_h^{k,\theta} \, \bm \mu_h \cdot \nabla \left( \overline{ \int_{0}^1 \gamma(m \bm M_h^{\xi}/ \rho_h^{\xi}) -1- \frac{ \left( \bm M_h^{\xi} \right) ^2}{c^2 \, \left(\rho_h^{\xi} \right)^2  \gamma(m \bm M_h^{\xi}/ \rho_h^{\xi})} \, d \xi} \right) c^2 \, dx \\ \notag
& \hspace{1cm}  + \frac{e}{m} \int_{\Omega_h} \rho_h^{k,\theta} \left( \bm E_h^{k+1/2} + \bm w_h \times \frac{\bm B_h^{k+1/2}}{c} \right)  \cdot \bm \mu_h \, dx \\ \notag
&\hspace{1cm} + \sum_{e \in \mathcal E_h} \int_{e } \left( \bm n \times \bm M_h^{\theta_3} \right)^{*} \cdot \llbracket \bm \mu_h \times \bm w_h \rrbracket \, ds\\ \notag
&\hspace{0cm} + \sum_{e \in \mathcal{E}_h} \int_{e}  \left(\rho_h^{k,\theta_2} \bm \mu_h \right)^{*} \cdot \left \llbracket \overline{ \int_0^1 \gamma(m \bm M_h^{\xi}/ \rho_h^{\xi}) -1- \frac{\left( \bm M_h^{\xi} \right)^2}{c^2 \, \left(\rho_h^{\xi} \right)^2 \gamma(m \bm M_h^{\xi}/ \rho_h^{\xi})} \, d \xi }  \right \rrbracket c^2 \, ds \hspace{3mm}\forall \bm \mu_h \in RT_k\\ \label{eq:dt_dgdiv_E}
& \int_{\Omega_h} \frac{\bm E_h^{k+1} - \bm E_h^k}{\Delta t} \cdot \bm \nu_h \, dx = c \int_{\Omega_h} \bm B_h^{k+1/2} \cdot (\nabla \times \bm \nu_h) \, dx - \frac{4 \pi e }{m} \int_{\Omega_h}  \left( \rho_h^{k,\theta} \bm w_h \cdot \bm \nu_h \right) \, dx \\ \notag
& \hspace{5cm} - 4 \pi e  \sum_p^N \sum_i^s w_p  \frac{( \bm A_{p,i}- \bm A_{p,i-1} ) }{\Delta t} \cdot \int_{0}^1 \bm  \nu_h(\bm X^{\xi}_{p,i}) d \xi   \hspace{1cm} \forall \bm \nu_h \in \mathring N_k
\\ \label{eq:dt_dgdiv_B}
& \int_{\Omega_h} \frac{\bm B_h^{k+1}- \bm B_h^k}{\Delta t} \cdot \bm \tau_h \, dx = - c \int_{\Omega_h} \bm \tau_h \cdot (\nabla \times \bm E_h^{k+1/2}) \, dx \hspace{1cm} \forall \bm \tau_h \in \mathring{RT}_k \\ \label{eq:dt_dgdiv_w}
& \text{with } \bm w_h = \overline{ \int_{0}^1 \frac{\bm M_h^{\xi}}{\rho_h^{\xi} \gamma(m \bm M_h^{\xi}/\rho_h^{\xi})} d \xi}
\end{align}
\begin{align} \label{eq:dt_dgdiv_X_k}
&\frac{\bm X_p^{k+1}- \bm X_p^k}{\Delta t}= \int_0^1 \frac{\bm U_p^{\xi}}{m \gamma(\bm U_p^{\xi})} d \xi \hspace{0.5cm} \text{with } \bm X_k \in \Omega_h \\ \label{eq:dt_dgdiv_U_k}
&\frac{\bm U_p^{k+1}-\bm U_p^k}{\Delta t} = e \left(  \sum_i^s \bm D_i \cdot \int_{0}^1 \bm E_h^{k+1/2}(\bm X^{\xi}_{p,i}) d \xi+ \frac{1}{m} \int_0^1 \frac{\bm U_p^{\xi}}{\gamma(\bm U_p^{\xi})} d \xi \times \frac{\bm B_h^{k+1/2}(\bm X_p^{k+1/2}, t)}{c} \right) 
\end{align}
In our numerical experiments, we set $\theta=\theta_2=\frac{1}{2}$.
We state the conservation properties of the temporal-discretization (\ref{eq:dt_dgdiv_rho}-\ref{eq:dt_dgdiv_U_k}) in the proposition below. We omit the proof which follows the proof in the Proposition \ref{prop:dt_qqnrt} but with the flux-terms canceled as in Proposition \ref{prop:dgdiv_proof}.

\begin{proposition} \label{prop:dt_dgdiv}
The temporal discretization in (\ref{eq:dt_dgdiv_rho}-\ref{eq:dt_dgdiv_U_k}) exactly conserves total fluid mass $\int_{\Omega_h} \rho_h^k \, dx$, total particle mass $\sum_p^N w_p$, total energy $\int_{\Omega_h} \rho_h^k (\gamma(m \bm M_h^k/ \rho_h^k)-1) c^2 dx + \sum_p^N w_p (\gamma(\bm U_p^k)-1)m c^2 +\frac{1}{8 \pi} \int_{\Omega_h} \left( (\bm E_h^k)^2 + \right.$ $\left. (\bm B_h^k)^2 \right) dx$, the weak Gauss's law  $ \int_{\Omega_h} \left[ \nabla_w \cdot \bm E^k -  4 \pi e \left(\frac{\rho^k_h}{m} + \sum_p^N w_p \delta(\bm x - \bm X_p^k) \right) \right] \, \phi_h \, dx$ for all $\phi_h \in \mathring Q_{k+1}$, and $\nabla \cdot \bm B_h^k$ constraint. 
\end{proposition}

\vspace{2.5mm}
\noindent \textbf{Solving the nonlinear system}. We use Picard's iteration to solve the non-linear system. At each iteration, we solve the linear system of equations using the conjugate gradient (CG) method with the Jacobi preconditioner.  Construction of a more robust solver is a topic of future work. For example, one can try to solve the fluid equations with the Newton-Raphson method, Maxwell's equations using a linear solver with the structure-preserving preconditioner from \cite{phillips2018scalable}, and the particles with Picard's iteration. These three sub-steps can be combined together using Strang-splitting and preserving the invariants; in much the same way as it was done for the particle-in-cell discretization in \cite{kormann_time}. 

\subsection{Explicit Runge-Kutta} \label{sec:explicit}

While the implicit AVF scheme proposed in the previous section features excellent conservation properties, it is expensive to evaluate due to its implicit nature. We therefore consider a SSP-RK propagator as an alternative. In particular, we will consider the order three SSP-RK scheme from \cite{Gottlieb2009} in our numerical experiments. SSP-RK methods consist of a convex combination of the explicit Euler steps. Here, we study the conservation properties of the explicit Euler discretization.
Let us first consider the case without particles.

\begin{proposition}
In the absence of particles, i.e. $N=0$, the explicit Euler method applied to (\ref{eq:qqnrt_rho}-\ref{eq:qqnrt_w}) and (\ref{eq:dgdiv_rho}-\ref{eq:dgdiv_w}) exactly conserves total mass $\int_{\Omega_h} \rho_h^k \, dx$, Gauss law  $\int_{\Omega} \left( \nabla_w \cdot \bm E^k - 4 \pi e \frac{\rho^k}{m} \right) \phi_h$ for all $\phi_h \in \mathring Q_{k+1}$, and $\nabla \cdot \bm B_h^k$ constraint. 
\end{proposition}
The proof proceeds as in Proposition \ref{prop:dt_qqnrt} and we omit it. 

When particles are present the explicit Euler does not preserve the Gauss law. To contain the Gauss’s law, we propose the following cleaning procedure.

\vspace{2.5mm}
\noindent \textbf{Exact Gauss law cleaning.} We adapt the idea of cleaning from \cite{brackbill}.
We consider functions $a_h$ and $b_h$ on the space $\mathring Q_{k+1}$. Let $\bm E_h \in \mathring N_k$ be the electric field that needs to be cleaned and $f$ be the right hand of the Gauss's law $\nabla  \cdot \bm E = f $. We solve two Poisson equations to compute the potentials $a_h$ and $b_h$
\begin{align}
\int_{\Omega_h} \nabla_w \cdot \nabla a_h \, \phi_h \, dx & = \int_{\Omega_h} \nabla_w \cdot \bm E_h \, \phi_h \, dx \hspace{1cm} \forall \phi_h \in \mathring Q_{k+1} \\
\int_{\Omega_h} \nabla_w \cdot \nabla b_h \, \phi_h \, dx & = \int_{\Omega_h} f \, \phi_h \, dx \hspace{1cm} \forall \phi_h \in \mathring Q_{k+1}
\end{align}
or equivalently 
\begin{align}
\int_{\Omega_h} \nabla a_h \cdot \nabla \phi_h \, dx & = \int_{\Omega_h} \bm E_h \cdot \nabla \phi_h \, dx \hspace{1cm} \forall \phi_h \in \mathring Q_{k+1} \\
\int_{\Omega_h} \nabla b_h \cdot \nabla \phi_h \, dx & = -\int_{\Omega_h} f \, \phi \, dx \hspace{1cm} \forall \phi_h \in \mathring Q_{k+1}
\end{align}
Next, we correct the electric field $\bm E_h$ by subtracting the potential associated with the polluted divergence and adding the potential corresponding to the correct divergence.
\begin{align} \label{eq:exact_gauss_cleaning}
\bm E'_h = \bm E_h - \nabla a_h + \nabla b_h
\end{align}
The Gauss's law holds
\begin{align}
\int_{\Omega_h} \nabla_w \cdot \bm E'_h \, \phi_h \, dx = \int_{\Omega_h} \left(  \underbrace{\nabla_w \cdot \bm E_h - \nabla_w \cdot \nabla a_h}_{=0} + \nabla_w \cdot \nabla b_h \right) \, \phi_h dx = \int_{\Omega_h} f \, \phi_h dx \hspace{1cm} \forall \phi_h \in \mathring Q_{k+1}
\end{align}

Finally, let us consider the question whether the energy is preserved up to the time-discretization order. Let us consider the case when the Gauss's cleaning is not used and there are only particles in the system. Let us assume that at each time-step the solution can be represented by $\bm a^k = \bm a_0(t) + \Delta t^p \bm a_1(t)$ where $\bm a_0$ is the exact value and $\bm a_1$ is a perturbation that is bounded and is independent of $\Delta t$. The energy difference due to particles is given by $\sum_p^N w_p (\gamma(\bm U_p^{k+1})-1) m \, c^2 - \sum_p^N w_p (\gamma(\bm U_{0,p})-1) m \, c^2$. We recall that $\gamma(\bm u) = \sqrt{1 + \frac{\bm u \cdot \bm u}{m ^2 c^2}}$. When $\bm u \cdot \bm u < m^2 c^2$ we can expand the $\gamma$-factor in a Taylor series up to some order $Z$.  After simple algebraic manipulations the energy error due to particles can be expressed as
\begin{align}
\sum_j^N w_j \sum_{i=1}^Z \binom{Z}{i}  \left[ \left( \frac{\bm U_{0,j} \cdot \bm U_{0,j} + 2  \Delta t^p  \bm U_{0,j} \cdot \bm U_{1, j} +\Delta t^{2p} \bm U_{1,j} \cdot \bm U_{1,j}}{m^2 \, c^2} \right)^i - \left( \frac{\bm U_{0,j} \cdot \bm U_{0,j}}{m^2 c^2} \right)^i \right] + \mathcal O \left( \left( \frac{\bm U^k_j \cdot \bm U^k_j}{m^2c^2} \right)^{Z+1} \right)
\end{align}
if we assume that $\frac{\bm U^k_j \cdot \bm U^k_j}{m^2c^2} \ll 1$, we can neglect the terms for $i \geq 2$. The leading term for $i=1$ then shows that energy is preserved up to the order $\mathcal O \left( \Delta t^p \right)$. If $\frac{\bm U^k \cdot \bm U^k}{m^2 \, c^2} > 1$, the Taylor expansion does not apply and we cannot analyze the order in this way. Since our numerical scheme is not bound-preserving, the convergence order of the energy error can be deteriorated for fast particles. The same conclusions hold for fluids and the hybrid case. Lastly, the case with Gauss's cleaning we consider numerically in the upcoming sections.

\section{Numerical examples} \label{sec:numerical_experiments}

We have implemented the developed discretizations based on the \texttt{deal.II} library \cite{dealii} and we investigate convergence and conservation properties in this section. Moreover, we compare a hybrid fluid-particle model to a full particle simulation for the example of a plasma wake simulation.

\subsection{Convergence of the fluid system.}

 To study the convergence of our discretization, we consider a three-dimensional domain $\Omega=[-1,1]^3$ and the following manufactured solutions
\begin{align} \label{eq:hbc_E}
\bm E(x,y,z,t) &= \begin{bmatrix}
-\sin(t) \, \cos(\pi x) \, \sin(\pi y) \, \sin(\pi z) \\
\sin(t) \, \sin(\pi x) \, \cos(\pi y) \, \sin(\pi z) \\
\sin(t) \, \sin(\pi x) \, \sin(\pi y) \, \cos(\pi z) 
\end{bmatrix} \\ \label{eq:hbc_B}
\bm B(x, y, z, t) &= \begin{bmatrix}
-\frac{1}{2} \cos(t) \, \sin(\pi x)\, \cos(\pi y) \, \cos(\pi z) \\
\frac{1}{4} \cos(t) \, \cos(\pi x)\, \sin(\pi y) \, \cos(\pi z) \\
\frac{1}{4} \cos(t) \, \cos(\pi x)\, \cos(\pi y) \, \sin(\pi z) 
\end{bmatrix} \\\label{eq:hbc_rho}
\rho &= 2 - \frac{m}{4 \, e} \sin(t) \, \sin(\pi x) \, \sin(\pi y) \, \sin(\pi z) \\ \label{eq:hbc_m}
\bm M &= \begin{bmatrix}
\sin(t) \, \sin(\pi x)\, \cos(\pi y) \, \cos(\pi z) \\
\sin(t) \, \cos(\pi x)\, \sin(\pi y) \, \cos(\pi z) \\
\sin(t) \, \cos(\pi x)\, \cos(\pi y) \, \sin(\pi z) 
\end{bmatrix} 
\end{align}
We add source terms $S_E$, $S_B$, $S_\rho$, $S_M$ to the right-hand side of the fluid equations (\ref{eq:deg_zero_hybrid_rho_relativ}-\ref{eq:deg_zero_hybrid_Mk_relativ}) and Maxwell equations (\ref{eq:maxwell_e}-\ref{eq:maxwell_b}). The source terms are constructed such that the manufactured solutions (\ref{eq:hbc_E} - \ref{eq:hbc_m}) are exact solutions. We note that the manufactured solutions satisfy $\bm n \times \bm E|_{\partial \Omega} = 0$, $\bm n \cdot \bm B|_{\partial \Omega}=0$, $\bm n \cdot \bm M|_{\partial \Omega}=0$, $\nabla \cdot \bm B = 0$, $\nabla \cdot \bm E = \frac{4 \pi e}{m} (\rho - 2)$. We use normalized physical quantities $e=-1$, $m=1$, $c=1$. We consider uniform trangulations with elements of size $h=2^{-i} h_0$ where $i={2,3,4,5}$ and $h_0=2$. The parameter $k$ that controls the degree of the finite element spaces was set to $\{0,1,2\}$. We use the implicit time-discretization based on the average-vector field method with a small time-step $\Delta t=0.00025$ to keep the temporal errors below the spatial errors. 
The average number of Picard's iterations needed for the simulations in Figure \ref{fig:hbc}  is 4. We measured the errors at $t=0.5$ using the $L^2(\Omega)$-norm that we denoted by $\| \cdot \|$.
\begin{figure}[H] 
\includegraphics[width=\textwidth]{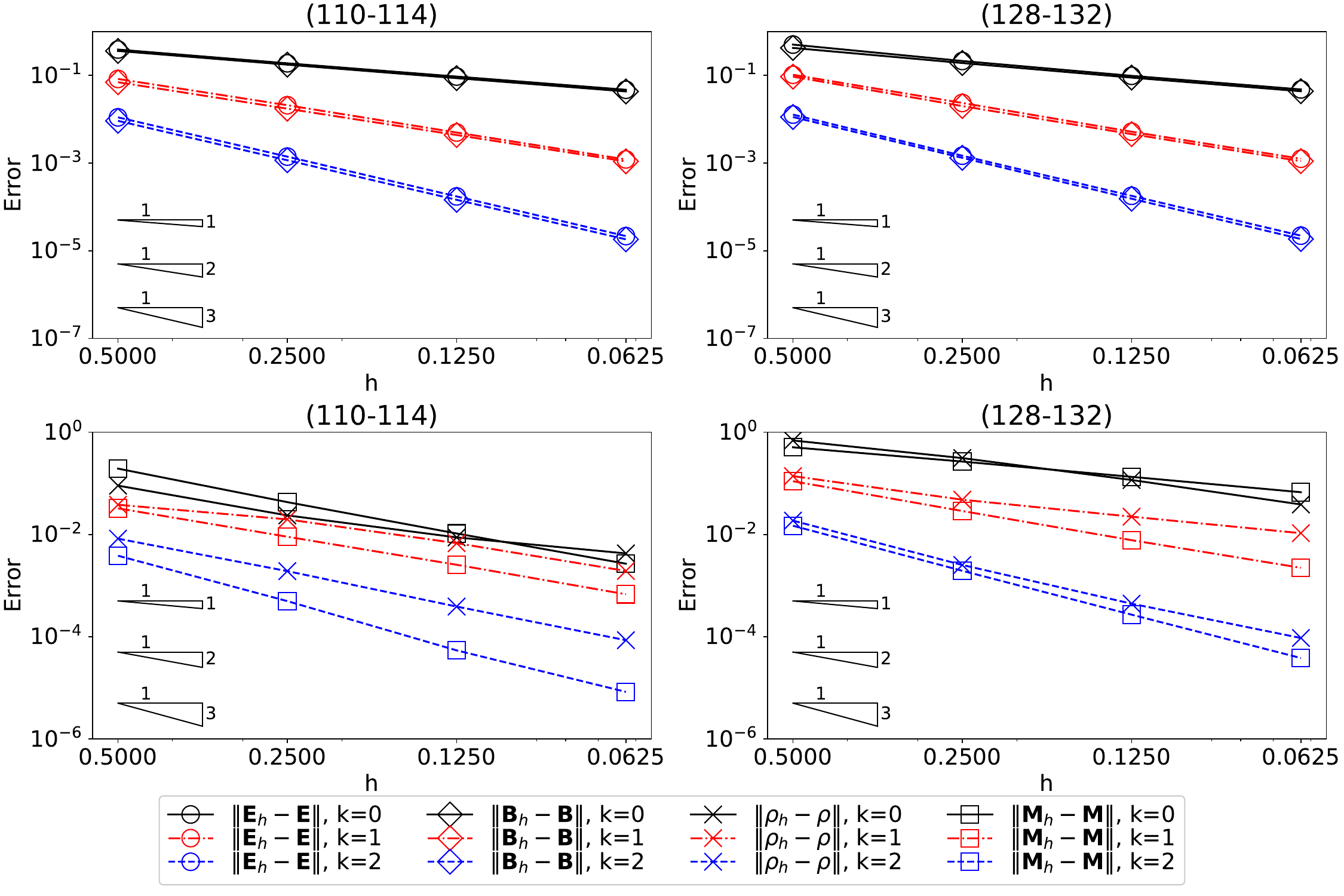}    
    \caption{$L^2$-errors in the electric field, magnetic field, fluid density, fluid momentum at time t=$0.5$ with the boundary conditions $\bm n \cdot \bm M|_{\partial \Omega}=0$, $\bm n \times \bm E|_{\partial \Omega} = 0$, $\bm n \cdot \bm B|_{\partial \Omega}=0$. The plots to the left are for the method (\ref{eq:dt_qqnrt_rho}-\ref{eq:dt_qqnrt_w}), the plots to the right are for the method (\ref{eq:dt_dgdiv_rho}-\ref{eq:dt_dgdiv_w}).}
    \label{fig:hbc}
\end{figure}
Both schemes show convergence of order $k+1$ in $\bm E$, $\bm B$, and $\bm M$. However, we observe an order reduction in $\rho$.

\subsection{Conservation properties} \label{sec:conservation}
We consider a three-dimensional domain $\Omega=[-1,1]^3$ and run numerical simulations to study the conservation properties of the schemes (\ref{eq:dt_qqnrt_rho} - \ref{eq:dt_qqnrt_U_k}) and (\ref{eq:dt_dgdiv_rho} - \ref{eq:dt_dgdiv_U_k}) in Section \ref{sec:time_integrators}. We use finite-elements with the degree parameter $k=0$.

We consider the following initial conditions
\begin{align} \label{eq:cs_ic_e_b}
& \bm E(x,y,z,t) = \begin{bmatrix}
- \cos(x \, y) \\
\cos(x \, z) \\
\sin(x \, y)
\end{bmatrix} \, 
\bm B(x,y,z,t) = \begin{bmatrix}
-\frac{1}{2} \sin(\pi \, x) \, \cos(\pi \, y) \, \cos(\pi \, z) \\
\frac{1}{4} \cos(\pi \, x) \, \sin(\pi \, y) \, \cos(\pi \, z) \\
\frac{1}{4} \cos(\pi \, x) \, \cos(\pi \, y) \, \sin(\pi \, z)
\end{bmatrix} \\ \label{eq:cs_ic_rho_m}
&\rho = 2 + \frac{m}{4 \pi e} y \, \sin(x \, y) \hspace{1cm}
\bm M(x,y,z,t) = \begin{bmatrix}
\frac{1}{4} \, \sin(\pi \, x) \\
\frac{1}{4} \, \sin(\pi \, y) \\
\frac{1}{4} \, \sin(\pi \, z)
\end{bmatrix}
\end{align}
We set the species mass and charge to unit values, $m=1$, $e=-1$. We use the speed of light $c=1$. On the boundary $\partial \Omega$, we assign constraints $\bm n \cdot \bm M|_{\partial \Omega}=0$, $\bm n \times \bm E|_{\partial \Omega}=0$, $\bm n \cdot \bm B|_{\partial \Omega}=0$. The initial condition for $\bm E$ does not satisfy the constraint $\bm n \times \bm E|_{\partial \Omega}=0$, and so we project it onto $\mathring N_k \subset N_k \cap \mathring H(\text{curl}, \Omega)$ using $L^2(\Omega)$ projection. We consider uniform mesh of size $h=0.125$. 

We draw the particle positions $\bm X_k$ from the Gaussian distribution
\begin{align}
f(x,y,z) =  \frac{1}{f_0} \begin{cases} \exp(-10 \, ( x^2 + y^2 + z^2 )) \hspace{0.5cm}  |x|<0.5 \text{ and } |y|<0.5 \text{ and } |z|<0.5 \\
 0  \hspace{3.1cm} \text{ otherwise}
\end{cases}
\end{align}
where $f_0$ is a normalizing factor such that $\int_{\Omega} f \, dx = 1$. This distribution ensures that particles can travel one quarter of the domain before escaping through the boundaries. The velocities of the particles, $\bm U_k$, are assigned randomly from $[0,1]$ with uniform probability. The total number of particles used is $10^5$. To keep the particle mass and current on the unit scale, we set the weights of each particle to $10^{-5}$.  

We calculate the errors as follows:
\begin{itemize}
\item mass error : $\frac{\text{mass}(t) - \text{mass}(0)}{\text{mass}(0)}$ where $\text{mass} = \int_\Omega \frac{\rho_h}{m} \, dx + \sum_k^N w_k $\\
\item energy error: $\frac{H(t) - H(0)}{H(0)}$ where $H(t)$ is given by \eqref{eq:relativistic_hamiltonian_hybrid_degree_zero_conserved}. \\
\item Gauss's law error: $ \max_i \int_{\Omega} \left[ \nabla \cdot \bm E - 4 \pi e  \left(\frac{\rho}{m} + \sum_k^N w_k \delta(\bm x - \bm X_k) - n_0 \right) \right] \, \phi_i \, dx$ where $\phi_i$ are the basis of $\mathring Q_{k+1}$ \\
\item divB constraint error: $\|\nabla \cdot \bm B_h \|_{L^2}$
\end{itemize}

We run until time $t=0.3$.
We start with the time-step $\Delta t=0.005$ and adaptively reduce it to make sure the fixed-point converges. At $t=0$, we use the exact Gauss's law cleaning proposed in Section \ref{sec:explicit} to clean the electric field and satisfy the Gauss's law. We approximate AVF integrals with four-point Gauss quadrature.

Figure \ref{fig:implicit} shows conservation of total mass, total energy, Gauss law and divB constraint. The errors are up to the tolerance of the CG solver set to $10^{-12}$.

\begin{figure}[H] 
\includegraphics[width=\textwidth]{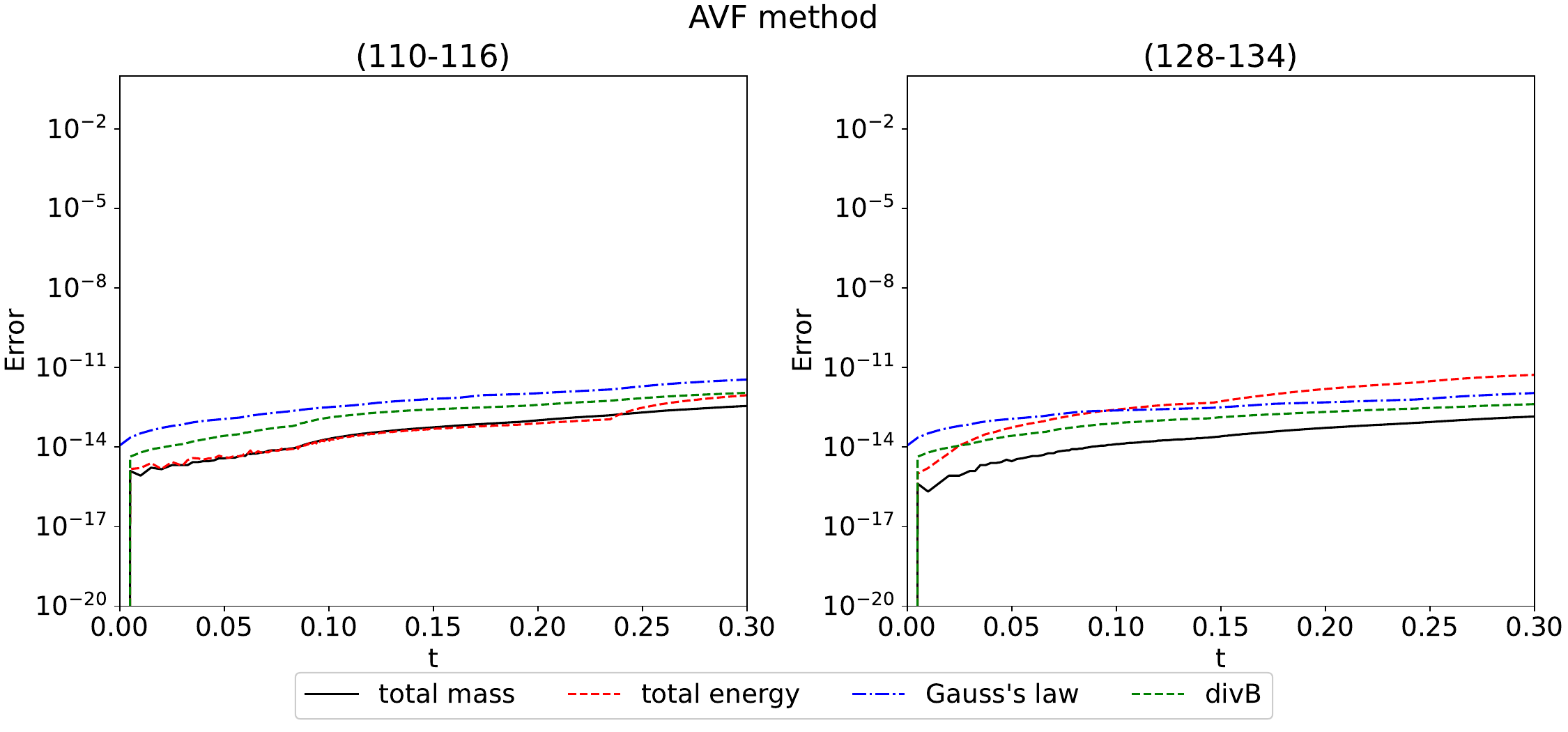}    
    \caption{Errors in the total mass, total energy, Gauss's law, and divB constraint for t=$[0,0.3]$. The plots to the left are for the method (\ref{eq:dt_qqnrt_rho} - \ref{eq:dt_qqnrt_U_k}), the plots to the right are for the method (\ref{eq:dt_dgdiv_rho} - \ref{eq:dt_dgdiv_U_k}).}
    \label{fig:implicit}
\end{figure}

Next, we repeat the experiment with the SSP-RK method of order three. We use the semi-discrete method (\ref{eq:qqnrt_rho}-\ref{eq:qqnrt_U_k}) and integrate it with the SSP-RK method in time. We consider the cases with $N=10^5$ and $N=0$ number of particles. We set the time-step size $\Delta t $ to $\{ 5 \times 10^{-5}, 1 \times 10^{-4}, 2 \times 10^{-4}, 5 \times 10^{-4} \}$. In view of the discussion in Section \ref{sec:explicit}, we set $c=10$ to make sure the particles velocities are below the speed of light. Figure \ref{fig:ssp_rk_c10} shows that for $N=10^5$ the scheme preserves total mass and the divB constraint. In the absence of particles, $N=0$, the scheme additionally preserves Gauss's law. The energy rate is also shown in both cases. 

\begin{figure}[H] 
\includegraphics[width=\textwidth]{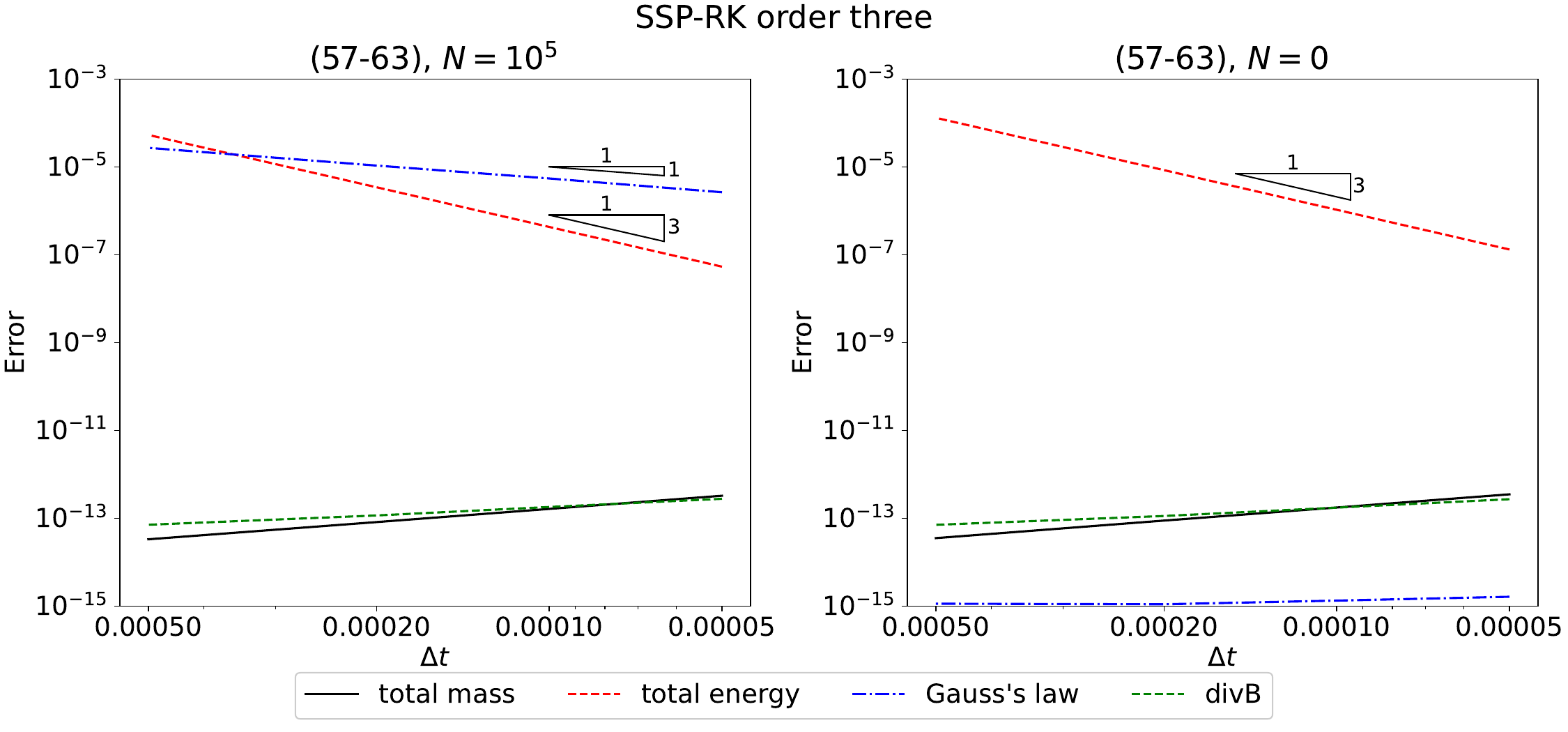}    
    \caption{Errors in the total mass, total energy, Gauss's law, and divB constraint for t=$[0,0.3]$ for the method (\ref{eq:qqnrt_rho}-\ref{eq:qqnrt_U_k}) integrated in time with SSP RK method of order three. The speed of light $c=10$. The plots to the left are for $N=10^5$ particles, while the plots to the right are for $N=0$, i.e. without particles}
    \label{fig:ssp_rk_c10}
\end{figure}

In Figure \ref{fig:ssp_rk_c10_with_gauss} we repeat the experiment but also apply the Gauss's cleaning every 100 steps. The hybrid scheme conserves total mass, Gauss's law, divB constraints while the total energy is conserved to first order. 

\begin{figure}[H] 
\centering
\includegraphics[width=0.5 \textwidth]{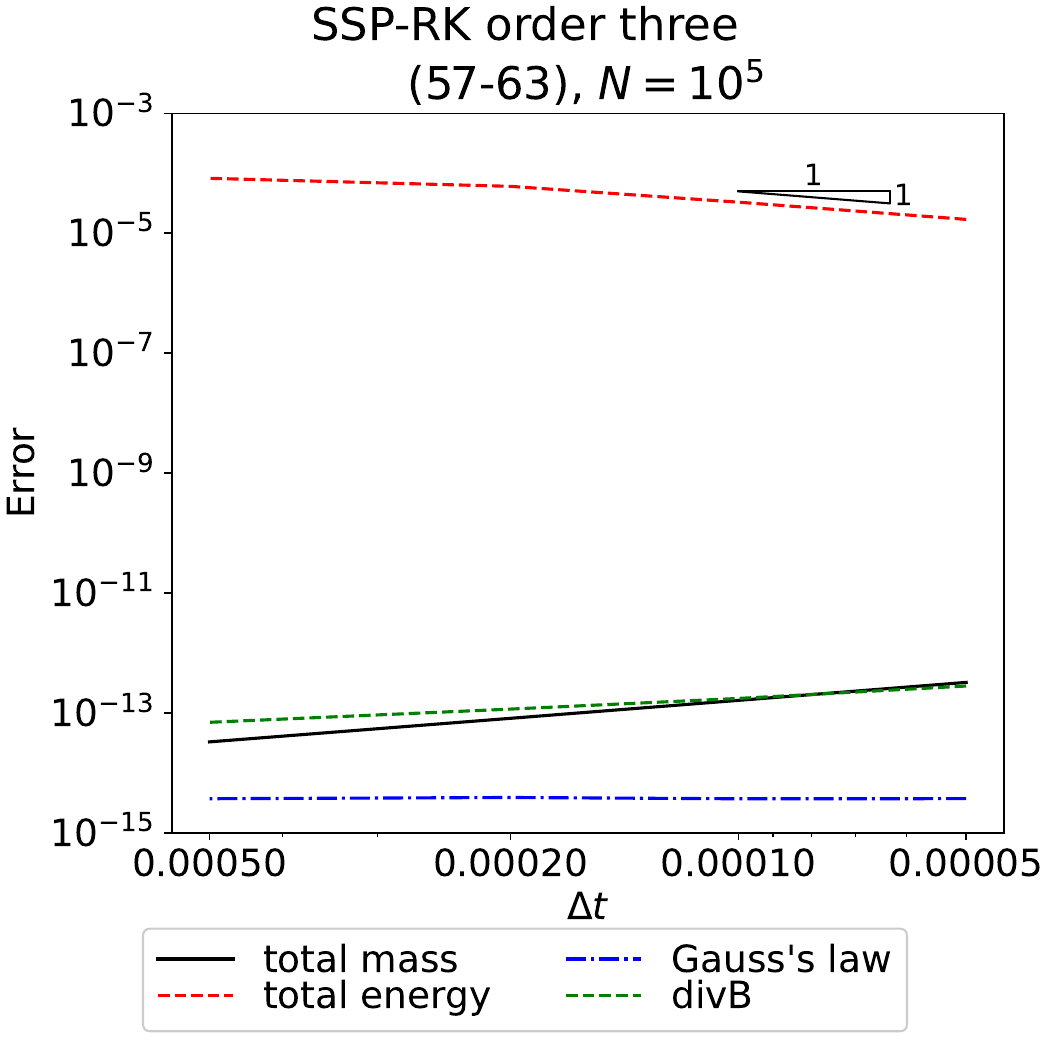}    
    \caption{Errors in the total mass, total energy, Gauss's law, and divB constraint for t=$[0,0.3]$ for the method (\ref{eq:qqnrt_rho}-\ref{eq:qqnrt_U_k}) integrated in time with SSP RK method of order three. The speed of light $c=10$. Gauss's cleaning is applied every 100 time-steps. The plots are for $N=10^5$ particles.}
    \label{fig:ssp_rk_c10_with_gauss}
\end{figure}

As discussed in Section \ref{sec:explicit}, the energy order may be lower than the order of the integrator when particles exceed the speed of light. This may occur because the model (\ref{eq:deg_zero_hybrid_rho_relativ}-\ref{eq:maxwell_div_b}) imposes no limit on the speed of fluid and particles. In Figure \ref{fig:ssp_rk} we set $c=1$ and repeat the experiment; we observe that the energy order reduces to second order for the case with particles. 

\begin{figure}[t] 
\includegraphics[width=\textwidth]{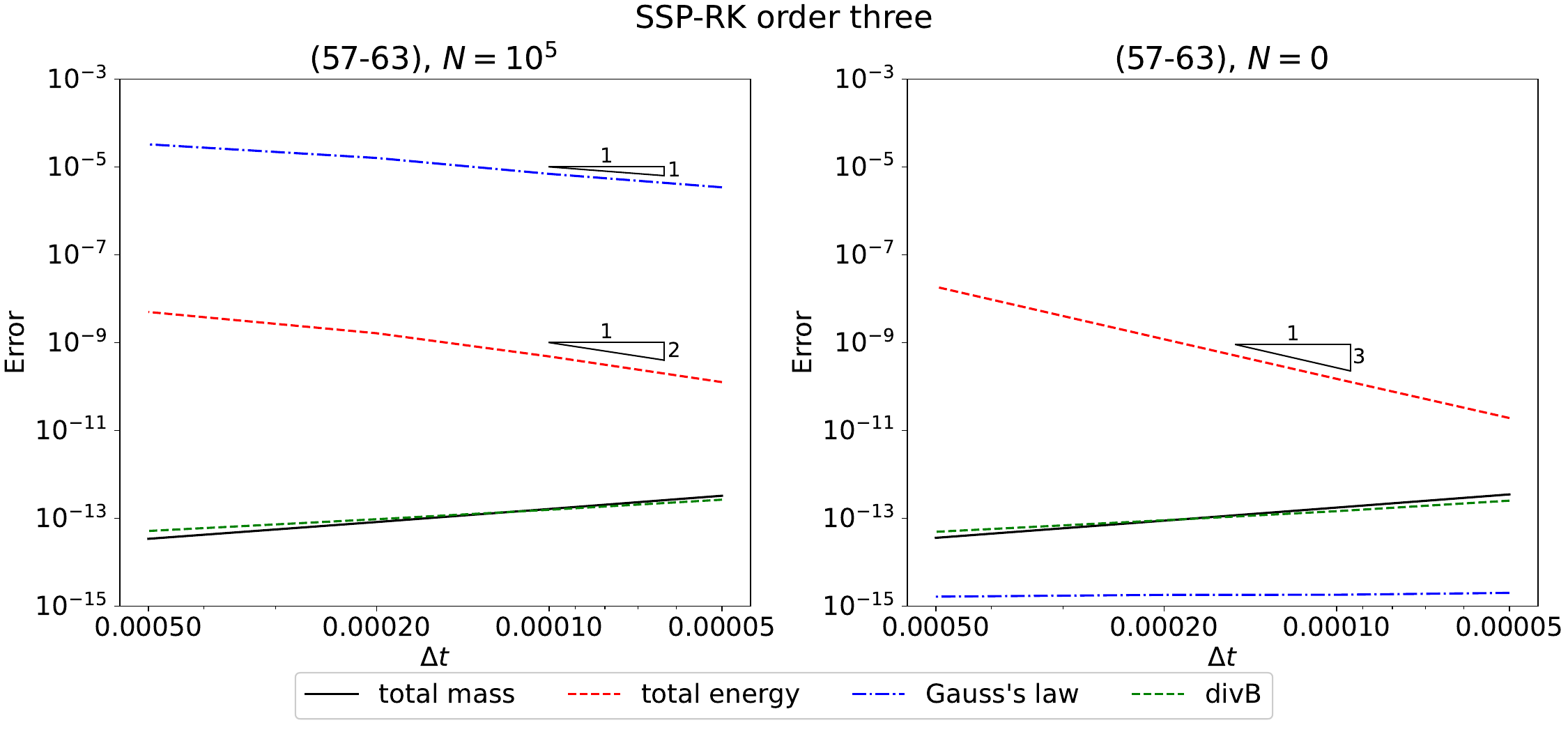}    
    \caption{Errors in the total mass, total energy, Gauss's law, and divB constraint for t=$[0,0.3]$ for the method (\ref{eq:qqnrt_rho}-\ref{eq:qqnrt_U_k}) integrated in time with SSP RK method of order three. The speed of light $c=1$. The plots to the left are for $N=10^5$ particles, while the plots to the right are for $N=0$, i.e. without particles.}
    \label{fig:ssp_rk}
\end{figure}

Table \ref{table:summary} summarizes conservation properties of the time-stepping methods from Section \ref{sec:time_integrators}.
\begin{table}[h!]
\begin{center}
\begin{tabular}{ c | c | c |  c | c | c}
Time-integrator & Mass & Energy & Gauss's law & divB & comments \\
\hline
 AVF (Section \ref{sec:avf}) & \tikzcmark & \tikzcmark & \tikzcmark & \tikzcmark & expensive \\ 
 SSP RK 3 with $N=0$ (Section \ref{sec:explicit}) & \tikzcmark & order 3  &\tikzcmark  & \tikzcmark & w.o. particles \\
   SSP RK 3 with $N>0$ (Section \ref{sec:explicit}) & \tikzcmark  & order 1 & \tikzcmark & \tikzcmark   &  w. particles; w. cleaning 
\end{tabular}
\end{center}
\caption{Conservation properties of time-stepping methods from Section \ref{sec:time_integrators}. Tick marks indicate exact conservation; otherwise the numerical convergence order with respect to the time-step is indicated.}
\label{table:summary}
\end{table}

\subsection{Plasma wake simulation}
In this section, we simulate a wake-field due to a plasma beam propagating at a speed close to the speed of light in a uniform background plasma. We employ the method (\ref{eq:qqnrt_rho}-\ref{eq:qqnrt_U_k}) with the SSP Runge-Kutta time-integrator of order three. We use particles to discretize the plasma beam and use the fluid description to model the background plasma. All simulations are performed in the laboratory frame of reference.

The setup is a modification of the setup in \cite{warpx_plasma_wake, warpx_general_paper}. We consider a three-dimensional box $[-200 \, \mu m, 200 \, \mu m]^3$ with $60 \times 60 \times 64$ cells. 

The plasma beam is positioned at $[-20 \mu m, 20 \mu m] \times [-20 \mu m, 20 \mu m] \times [-150 \mu m, -100 \mu m]$ with the initial velocity $[0, 0, 2.9 \times 10^8 \, m/s]$. For the beam we use in total $8 \times 10^4$ particles with the weight $w = 10^4$. This corresponds to the plasma density of $10^{22} \, m^{-3}$. The particles are distributed uniformly.

At time $t=0$, we assume the background plasma has the density $n(0) = 10^{22} \, m^{-3}$. The background plasma is assumed to be at rest with zero momentum $\bm M(0)=0$. 

We set the electric field $\bm E(0)=0$ and apply Gauss's law cleaning at $t=0$; for the magnetic field we assume $\bm B(0)=0$. We model only the electrons and assume the ions are stationary.

We use periodic boundary conditions for the fields $\rho, \bm M, \bm E, \bm B$ on the faces $x=\pm 200 \, \mu m$ and $y=\pm 200 \, \mu m$. On the faces $z=\pm 200 \, \mu m$ , we use the perfect-electric conductor boundary for the electric and magnetic fields: $\bm n \times \bm E = 0$ and $\bm n \times \bm B = 0$.  For the density $\rho$ and for the momentum $\bm M$, we assume zero flux boundary $\bm n \cdot \bm M =0$. We set the time-step to  $\Delta t=2 \times 10^{-17} s$.

To compare the results, we consider a reference simulation where the background plasma is discretized with the particle-in-cell method (\ref{eq:qqnrt_E}-\ref{eq:qqnrt_U_k}) using $64 \times 10^6$ particles with the weight $w=10^4$. We use SSP-RK method of order three with the time-step $\Delta t=10^{-17} s$. For the visualization of the density $\rho$ we project the particles on discontinuous Galerkin space $DG_0$. For particles, we use periodic boundaries  on the faces $x=\pm 200 \, \mu m$ and $y=\pm 200 \, \mu m$ and open boundaries on the faces $z=\pm 200 \, \mu m$.

We clean the Gauss's law every $100$ time-steps according to \eqref{eq:exact_gauss_cleaning}. 

In Figure \ref{fig:plasma_wake} we compare the $E_z$ component the fluid simulation with the reference particle simulation. The upper row (Figure \ref{fig:first}-\ref{fig:third}) shows the evolution of the plasma wake for the hybrid formulation. The lower row (Figure \ref{fig:fourth}-\ref{fig:sixth}) shows the evolution of the plasma wake for the reference particle simulation.

\begin{figure}[H]
\centering
\begin{subfigure}{0.32\textwidth}
    \includegraphics[width=\textwidth]{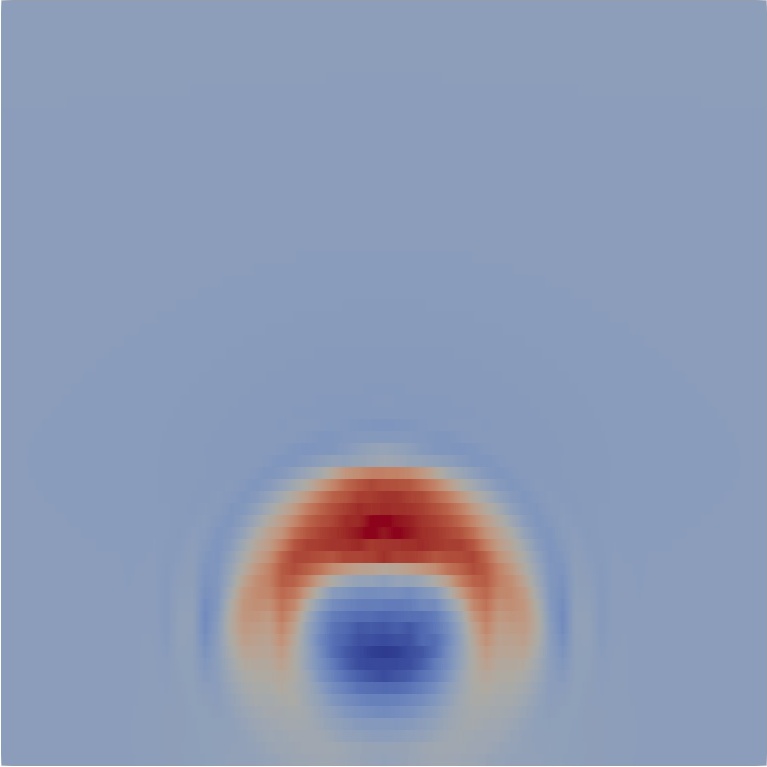}
    \caption{$t=2 \times 10^{-13} s$}
    \label{fig:first}
\end{subfigure}
\hfill
\begin{subfigure}{0.32\textwidth}
    \includegraphics[width=\textwidth]{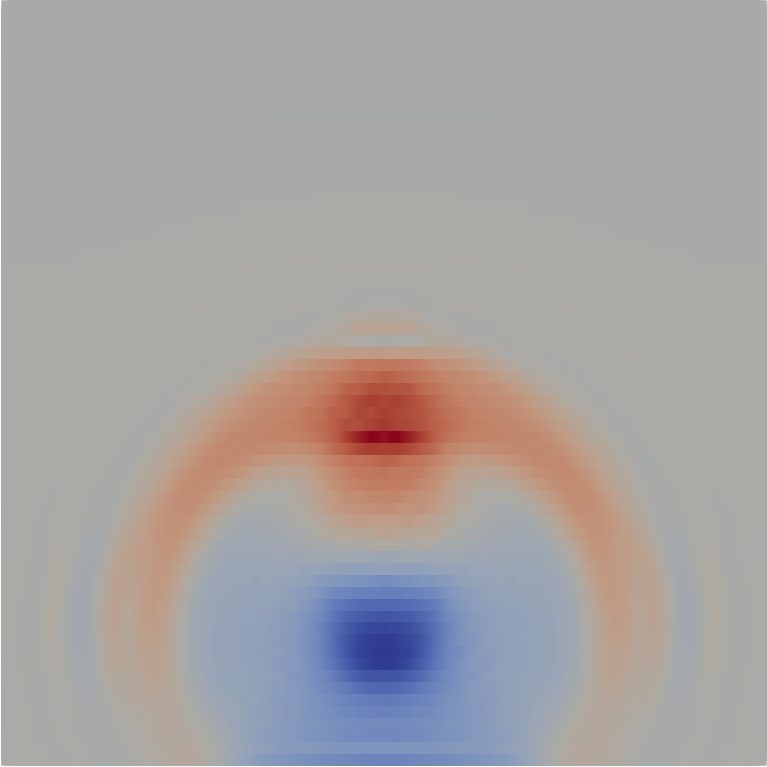}
    \caption{$t=4 \times 10^{-13} s$}
    \label{fig:second}
\end{subfigure}
\hfill
\begin{subfigure}{0.32\textwidth}
    \includegraphics[width=\textwidth]{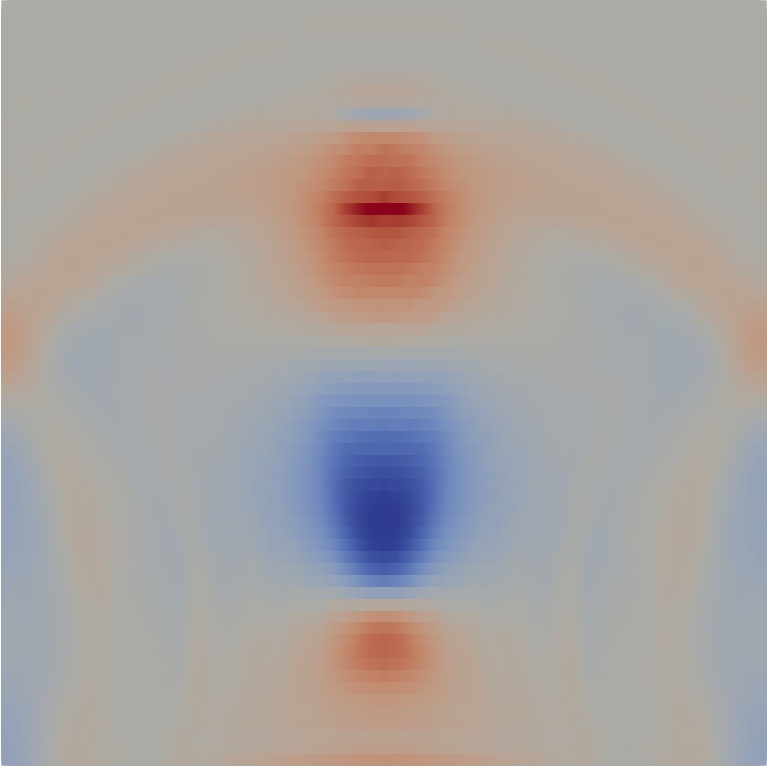}
    \caption{$t=8 \times 10^{-13} s$}
    \label{fig:third}
\end{subfigure}
\centering
\begin{subfigure}{0.32\textwidth}
    \includegraphics[width=\textwidth]{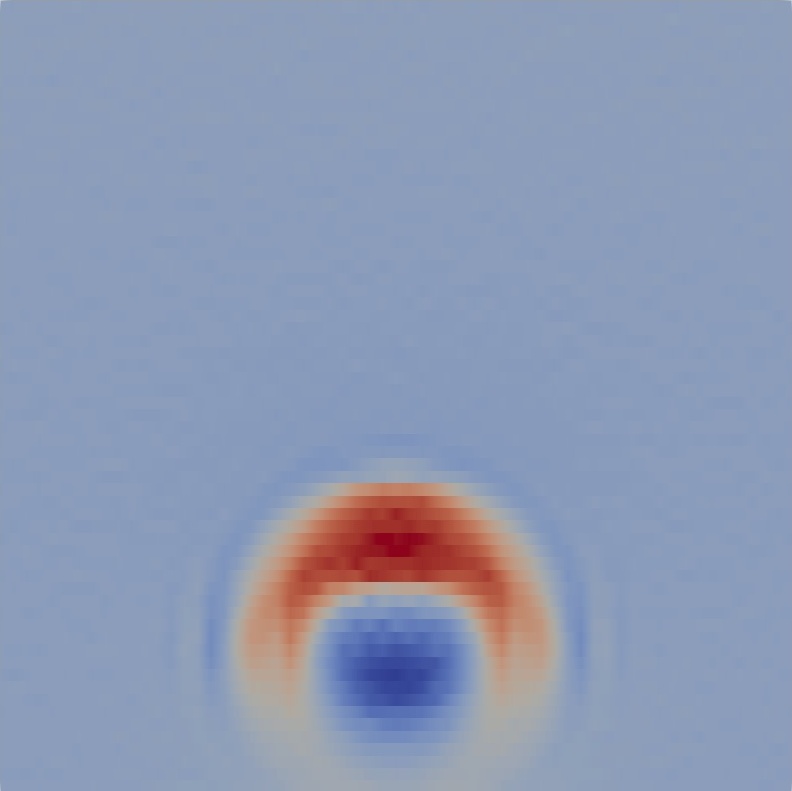}
    \caption{$t=2 \times 10^{-13} s$}
    \label{fig:fourth}
\end{subfigure}
\hfill
\begin{subfigure}{0.32\textwidth}
    \includegraphics[width=\textwidth]{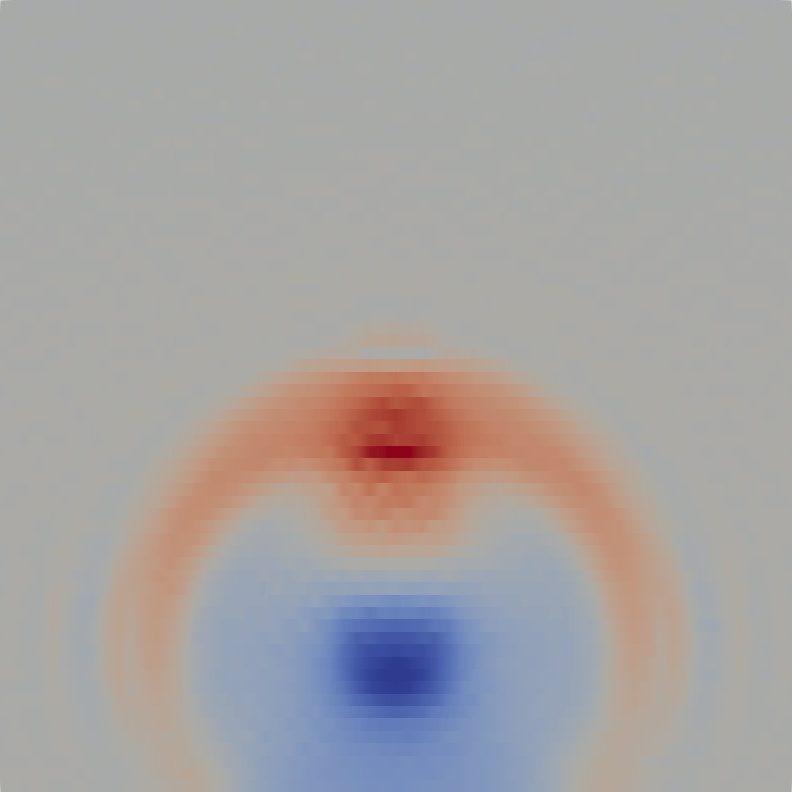}
    \caption{$t=4 \times 10^{-13} s$}
    \label{fig:fifth}
\end{subfigure}
\hfill
\begin{subfigure}{0.32\textwidth}
    \includegraphics[width=\textwidth]{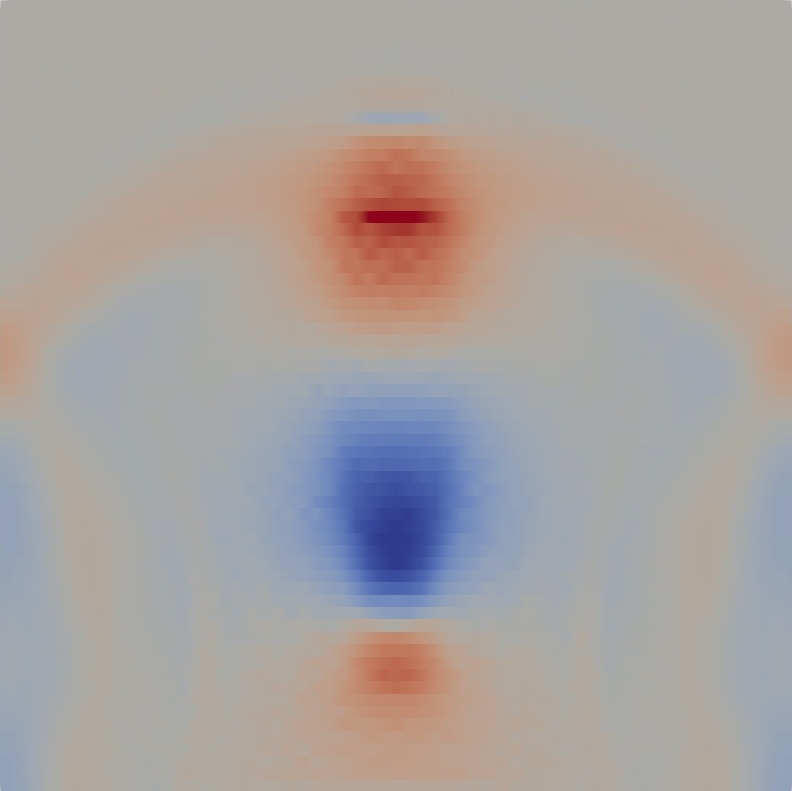}
    \caption{$t=8 \times 10^{-13} s$}
    \label{fig:sixth}
\end{subfigure}

\caption{Plasma-wake generated by the hybrid formulation (\ref{eq:qqnrt_rho}-\ref{eq:qqnrt_U_k})  in Figures (a)-(b) versus the reference particle discretization in Figures (d)-(f). We show the slices through the origin, perpendicular to the $x$-axis. Illustrated is the electric field component $E_z$ at $t=2 \times 10^{-13}$, $t=4 \times 10^{-13}$, $t=8 \times 10^{-13}$. $Z$-coordinate is pointing upwards.}
\label{fig:plasma_wake}
\end{figure}

In Figure \ref{fig:rho_plasma_wake} we compare the density $\rho$ of the fluid simulation with the reference particle simulation. Notice the slight perturbation in Figure \ref{fig:rho_second} near the lower boundary $z=-200 \, \mu m$ that is seen in the hybrid model. This perturbation is excited by the beam and disappears in Figure \ref{fig:rho_third} as the beam travels away from the boundary. This perturbation is not seen in the reference particle simulation as we use open boundaries at $z=-200 \, \mu m$ that allow particles to escape.

\begin{figure}[H]
\centering
\begin{subfigure}{0.32\textwidth}
    \includegraphics[width=\textwidth]{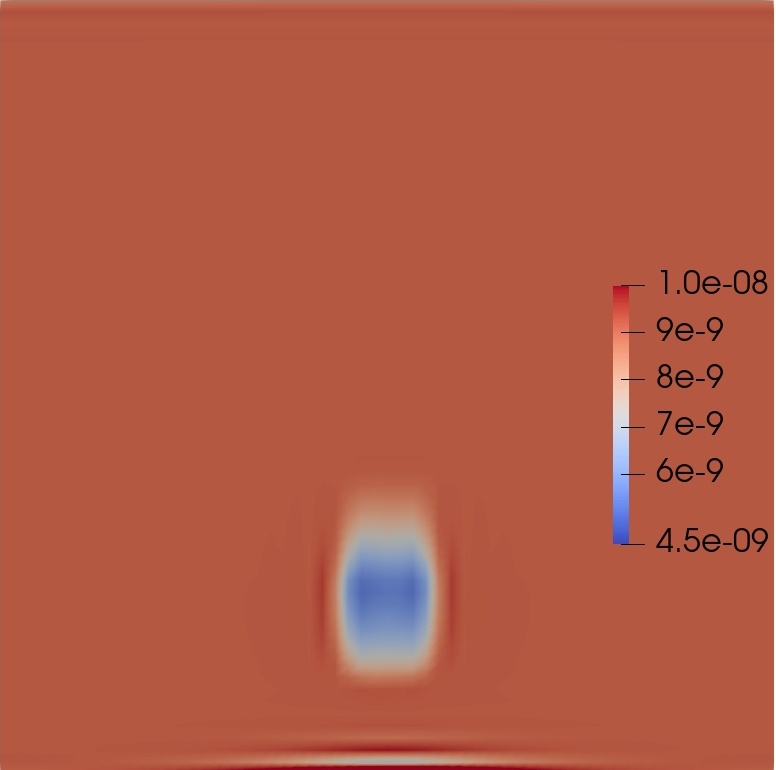}
    \caption{$t=2 \times 10^{-13} s$}
    \label{fig:rho_first}
\end{subfigure}
\hfill
\begin{subfigure}{0.32\textwidth}
    \includegraphics[width=\textwidth]{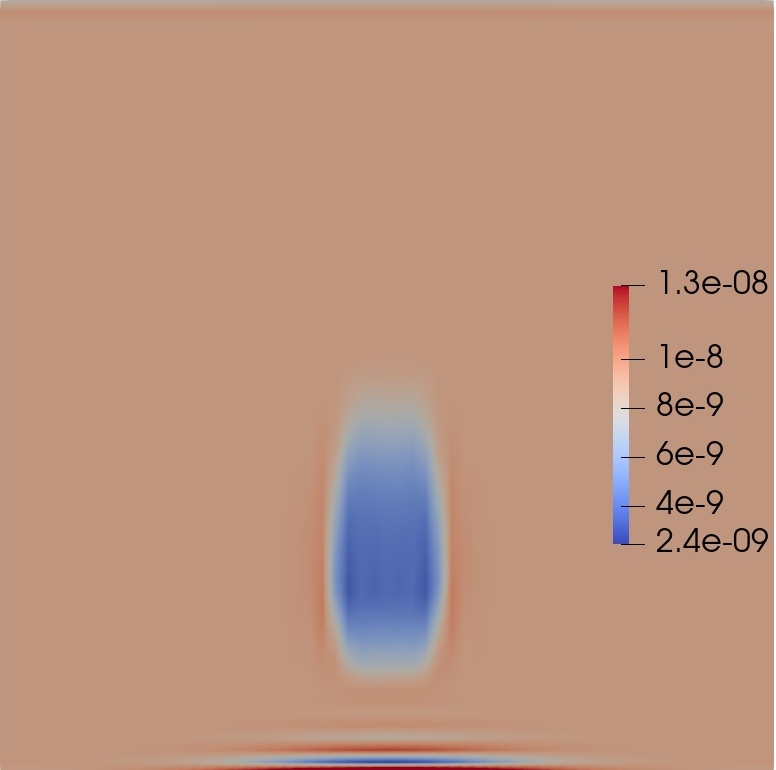}
    \caption{$t=4 \times 10^{-13} s$}
    \label{fig:rho_second}
\end{subfigure}
\hfill
\begin{subfigure}{0.32\textwidth}
    \includegraphics[width=\textwidth]{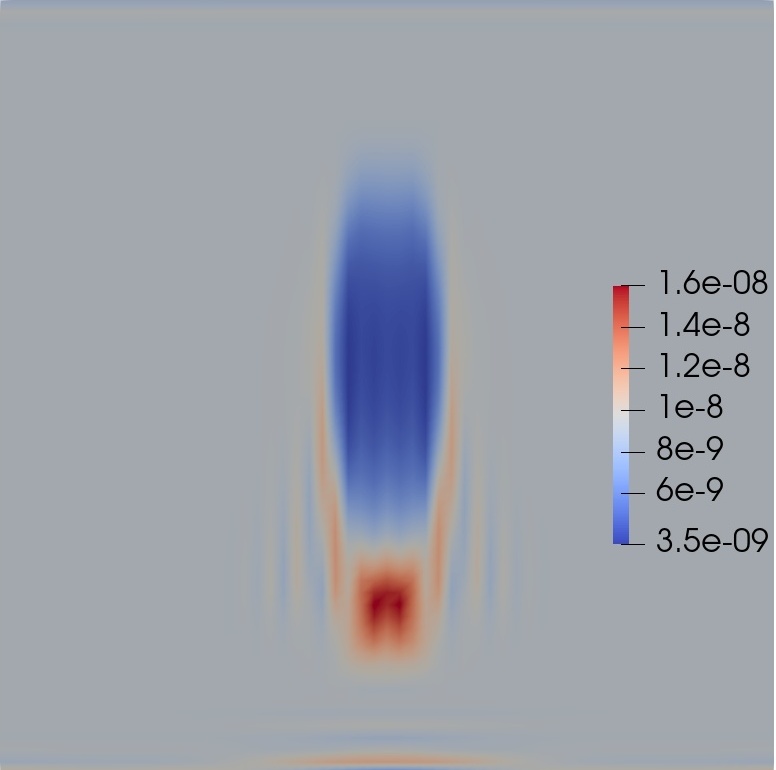}
    \caption{$t=8 \times 10^{-13} s$}
    \label{fig:rho_third}
\end{subfigure}
\centering
\begin{subfigure}{0.32\textwidth}
    \includegraphics[width=\textwidth]{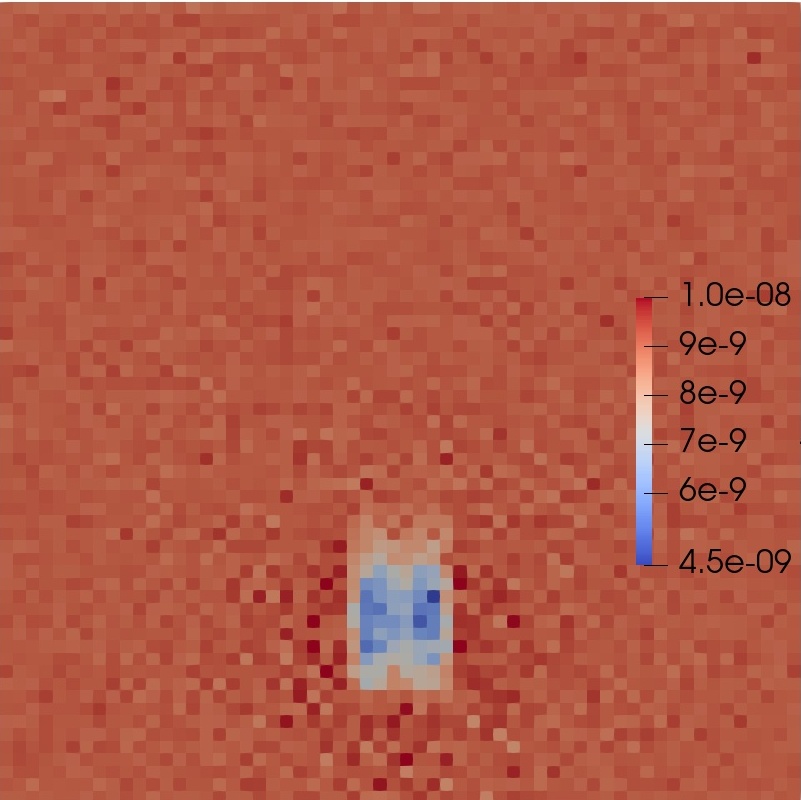}
    \caption{$t=2 \times 10^{-13} s$}
    \label{fig:rho_fourth}
\end{subfigure}
\hfill
\begin{subfigure}{0.32\textwidth}
    \includegraphics[width=\textwidth]{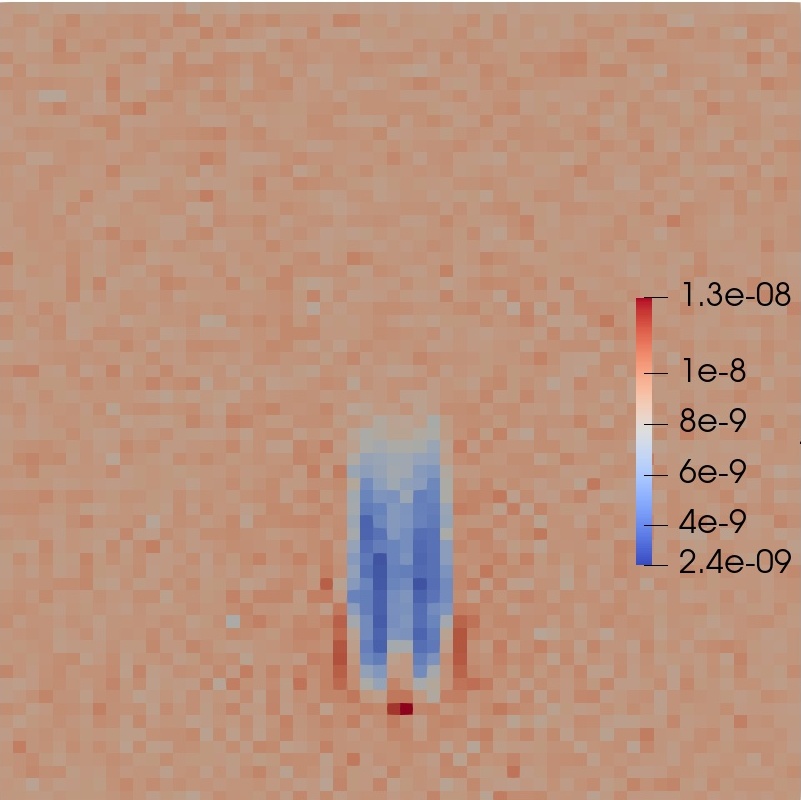}
    \caption{$t=4 \times 10^{-13} s$}
    \label{fig:rho_fifth}
\end{subfigure}
\hfill
\begin{subfigure}{0.32\textwidth}
    \includegraphics[width=\textwidth]{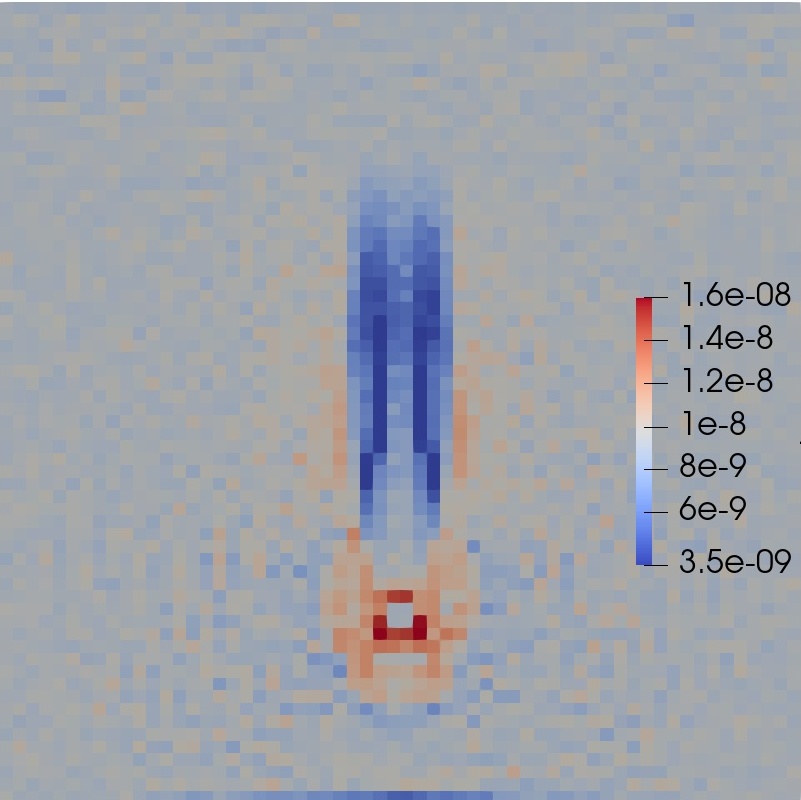}
    \caption{$t=8 \times 10^{-13} s$}
    \label{fig:rho_sixth}
\end{subfigure}

\caption{Plasma-wake generated by the hybrid formulation (\ref{eq:qqnrt_rho}-\ref{eq:qqnrt_U_k})  in Figures (a)-(b) versus the reference particle discretization in Figures (d)-(f). The reference simulation plots are obtained by projecting the particles weights onto the discontinuous Galerkin space $DG_0$. We show the slices through the origin, perpendicular to the $x$-axis. Illustrated is the density $\rho$ at $t=2 \times 10^{-13}$, $t=4 \times 10^{-13}$, $t=8 \times 10^{-13}$. $Z$-coordinate is pointing upwards. }
\label{fig:rho_plasma_wake}
\end{figure}

We note that the formulation developed in this work does not ensure the positivity of the density and cannot be used for the simulation of the blow out regime when the density goes to zero. Future work should incorporate the positivity limiter, such as the artifical viscosity \cite{dao, popov}, to keep the density positive.

\section{Acknowledgments}
Martin Kronbichler is acknowledged for his help with the \texttt{deal.II} library and Grant R. Johnson for pointing out suitable numerical experiments. We acknowledge financial support from the Deutsche Forschungsgemeinschaft (DFG) via the Collaborative Research Center SFB1491 Cosmic Interacting Matters—From Source to Signal. Calculations for this publication were performed on the HPC cluster Elysium of the Ruhr University Bochum, subsidised by the DFG (INST 213/1055-1).

\bibliographystyle{plain}

\bibliography{paper_clean}

\begin{thebibliography}{10}

\bibitem{warpx_plasma_wake}
Warpx documentation 23.11. {B}eam-{D}riven {W}akefield {A}cceleration of
  {E}lectrons.
\newblock Accessed on September 27, 2025.

\bibitem{dealii}
Pasquale~C. Africa, Daniel Arndt, Wolfgang Bangerth, Bruno Blais, Marc Fehling,
  Rene Gassmöller, Timo Heister, Luca Heltai, Sebastian Kinnewig, Martin
  Kronbichler, Matthias Maier, Peter Munch, Magdalena Schreter-Fleischhacker,
  Jan~P. Thiele, Bruno Turcksin, David Wells, and Vladimir Yushutin.
\newblock {The deal.II library, Version 9.6}.
\newblock {\em J. Numer. Math.}, 32(4):369--380, 2024.

\bibitem{arnold2018finite}
Douglas~N Arnold.
\newblock {\em Finite Element Exterior Calculus}.
\newblock SIAM, 2018.

\bibitem{bera2015fluid}
Ratan~Kumar Bera, Sudip Sengupta, and Amita Das.
\newblock Fluid simulation of relativistic electron beam driven wakefield in a
  cold plasma.
\newblock {\em Physics of Plasmas}, 22(7), 2015.

\bibitem{brackbill}
J.U Brackbill and D.C Barnes.
\newblock The {E}ffect of {N}onzero $\nabla \cdot$ {B} on the numerical
  solution of the magnetohydrodynamic equations.
\newblock {\em Journal of Computational Physics}, 35(3):426--430, 1980.

\bibitem{discretegradient}
C.J. Budd, A.~Iserles, Robert~I. McLachlan, G.~R.~W. Quispel, and Nicolas
  Robidoux.
\newblock Geometric integration using discrete gradients.
\newblock {\em Philosophical Transactions of the Royal Society of London.
  Series A: Mathematical, Physical and Engineering Sciences},
  357(1754):1021--1045, 1999.

\bibitem{gradshyperbolicity}
Zhenning Cai, Yuwei Fan, and Ruo Li.
\newblock {On Hyperbolicity of 13-Moment System}.
\newblock {\em Kinetic and Related Models}, 7:415 -- 432, 07 2014.

\bibitem{CARLIER2025113647}
Valentin Carlier and Martin Campos-Pinto.
\newblock Variational discretizations of ideal magnetohydrodynamics in smooth
  regime using structure-preserving finite elements.
\newblock {\em Journal of Computational Physics}, 523:113647, 2025.

\bibitem{dao}
Tuan~Anh Dao and Murtazo Nazarov.
\newblock {Monolithic parabolic regularization of the MHD equations and entropy
  principles}.
\newblock {\em Computer Methods in Applied Mechanics and Engineering},
  398:115269, 2022.

\bibitem{warpx_general_paper}
Luca Fedeli, Axel Huebl, France Boillod-Cerneux, Thomas Clark, Kevin Gott,
  Conrad Hillairet, Stephan Jaure, Adrien Leblanc, Rémi Lehe, Andrew Myers,
  Christelle Piechurski, Mitsuhisa Sato, Neïl Zaim, Weiqun Zhang, Jean-Luc
  Vay, and Henri Vincenti.
\newblock Pushing the frontier in the design of laser-based electron
  accelerators with groundbreaking mesh-refined particle-in-cell simulations on
  exascale-class supercomputers.
\newblock In {\em SC22: International Conference for High Performance
  Computing, Networking, Storage and Analysis}, pages 1--12, 2022.

\bibitem{gawlik2020conservative}
Evan~S Gawlik and Fran{\c{c}}ois Gay-Balmaz.
\newblock A conservative finite element method for the incompressible {E}uler
  equations with variable density.
\newblock {\em Journal of Computational Physics}, 412:109439, 2020.

\bibitem{gawlik2021variationallie}
Evan~S Gawlik and Fran{\c{c}}ois Gay-Balmaz.
\newblock {A Variational Finite Element Discretization of Compressible Flow}.
\newblock {\em Foundations of Computational Mathematics}, 21(4):961--1001,
  2021.

\bibitem{gawlik2021structure}
Evan~S Gawlik and Fran{\c{c}}ois Gay-Balmaz.
\newblock A structure-preserving finite element method for compressible ideal
  and resistive magnetohydrodynamics.
\newblock {\em Journal of Plasma Physics}, 87(5):835870501, 2021.

\bibitem{gawlik2022finite}
Evan~S Gawlik and Fran{\c{c}}ois Gay-Balmaz.
\newblock A finite element method for {MHD} that preserves energy,
  cross-helicity, magnetic helicity, incompressibility, and div {B}= 0.
\newblock {\em Journal of Computational Physics}, 450:110847, 2022.

\bibitem{Gottlieb2009}
Sigal Gottlieb, David~I. Ketcheson, and Chi-Wang Shu.
\newblock High order strong stability preserving time discretizations.
\newblock {\em Journal of Scientific Computing}, 38(3):251--289, 2009.

\bibitem{grad1949kinetic}
Harold Grad.
\newblock On the kinetic theory of rarefied gases.
\newblock {\em Communications on Pure and Applied Mathematics}, 2(4):331--407,
  1949.

\bibitem{popov}
Jean-Luc Guermond and Bojan Popov.
\newblock {Viscous Regularization of the Euler Equations and Entropy
  Principles}.
\newblock {\em SIAM Journal on Applied Mathematics}, 74(2):284--305, 2014.

\bibitem{holderied2021mhd}
Florian Holderied, Stefan Possanner, and Xin Wang.
\newblock {MHD-kinetic hybrid code based on structure-preserving finite
  elements with particles-in-cell}.
\newblock {\em Journal of Computational Physics}, 433:110143, 2021.

\bibitem{warpx}
Grant~R Johnson.
\newblock Warpx documentation 23.11. {C}old {R}elativistic {F}luid {M}odel.
\newblock Accessed on September 27, 2025.

\bibitem{kormann_time}
Katharina Kormann and Eric Sonnendrücker.
\newblock {Energy-conserving time propagation for a structure-preserving
  particle-in-cell Vlasov–Maxwell solver}.
\newblock {\em Journal of Computational Physics}, 425:109890, 2021.

\bibitem{gempic}
Michael Kraus, Katharina Kormann, Philip J. Morrison, and Eric Sonnendrücker.
\newblock {GEMPIC}: geometric electromagnetic particle-in-cell methods.
\newblock {\em Journal of Plasma Physics}, 83(4):905830401, 2017.

\bibitem{moralsancoz}
Elena {Moral Sánchez}, Martin {Campos Pinto}, Yaman Güçlü, and Omar Maj.
\newblock Time-splitting methods for the cold-plasma model using finite element
  exterior calculus.
\newblock {\em Journal of Computational Physics}, 541:114305, 2025.

\bibitem{morrison1998hamiltonian}
Philip~J Morrison.
\newblock Hamiltonian description of the ideal fluid.
\newblock {\em Reviews of modern physics}, 70(2):467, 1998.

\bibitem{phillips2018scalable}
Edward~G Phillips, John~N Shadid, and Eric~C Cyr.
\newblock Scalable preconditioners for structure preserving discretizations of
  maxwell equations in first order form.
\newblock {\em SIAM Journal on Scientific Computing}, 40(3):B723--B742, 2018.

\bibitem{average_vector}
G~R~W Quispel and D~I McLaren.
\newblock A new class of energy-preserving numerical integration methods.
\newblock {\em Journal of Physics A: Mathematical and Theoretical},
  41(4):045206, jan 2008.

\bibitem{tronci2010}
Cesare Tronci.
\newblock Hamiltonian approach to hybrid plasma models.
\newblock {\em Journal of Physics A: Mathematical and Theoretical},
  43(37):375501, jul 2010.

\end{thebibliography}

\begin{appendix}

\section{Eigenvalues} \label{appendix_A}
Here we calculate eigenvalues of (\ref{eq:deg_zero_hybrid_rho_relativ}-\ref{eq:maxwell_b}) which physically represent the speeds present in the model. 


The relativistic fluid system can be written in terms of the primitive variables $n$ and $\bm v$. Physically, these variables represent the particle density and the center of the relativistic Vlasov-Maxwell distribution in the phase-space, i.e., $f(\bm x, \bm p, t) = n(\bm x, t) \, \delta(\bm p - m \bm v(\bm x, t))$. Conserved variables $\rho, \bm M$ can be related to the primitive variables according to: $\rho = m \, n$ and $\bm M = \rho \bm v$.

Using this relationship, equations (\ref{eq:deg_zero_hybrid_rho_relativ}-\ref{eq:deg_zero_hybrid_Mk_relativ}) can be written in terms of the primitive variables as follows
\begin{align} \label{eq:}
\partial_t n & = - \nabla \cdot \left( \frac{n \, \bm v}{\gamma(m \bm v)} \right) = - \frac{n}{\gamma(m \bm v)} \nabla \cdot \bm v - \frac{\bm v}{\gamma(m \bm v)} \cdot \nabla n + \frac{n}{c^2 \, \gamma(m \bm v)^3} \bm v \otimes \bm v : \nabla \bm v \\
\partial_t \bm v &= - \frac{\bm v}{\gamma(m \bm v)} \cdot \nabla \bm v + \frac{e}{m} \left( \bm E + \frac{\bm v}{c \gamma(m \bm v)} \times \bm B \right)
\end{align}
To calculate eigenvalues, we write the system as follows
\begin{align}
\partial_t \bm U = \bm A_x \partial_x \bm U + \bm A_y \partial_y \bm U + \bm A_z \partial_z \bm U + \bm S(\bm U)
\end{align}
where $\bm U= (n, v_x, v_y, v_z)^\top$, $\bm S(\bm U)$ is the source containing Lorentz force term with electric and magnetic fields, and $\bm A_x, \bm A_y, \bm Az$ are the flux Jacobian matrices
\begin{align}
\bm A_x = \begin{bmatrix}
-\frac{v_x}{\gamma(m \bm v)} & - \frac{n}{\gamma(m \bm v)} + \frac{n}{c^2 \, \gamma(m \bm v)^3} v_x v_x & \frac{n}{c^2 \, \gamma(m \bm v)^3} v_x v_y & \frac{n}{c^2 \, \gamma(m \bm v)^3} v_x v_z \\ 
0 & -\frac{v_x}{\gamma(m \bm v)} & 0 & 0 \\
0 & 0 & -\frac{v_x}{\gamma(m \bm v)}  & 0 \\
0 & 0 & 0 & -\frac{v_x}{\gamma(m \bm v)} 
\end{bmatrix} \\
\bm A_y = \begin{bmatrix}
-\frac{v_y}{\gamma(m \bm v)} & \frac{n}{c^2 \, \gamma(m \bm v)^3} v_y v_x & - \frac{n}{\gamma(m \bm v)}  + \frac{n}{c^2 \, \gamma(m \bm v)^3} v_y v_y & \frac{n}{c^2 \, \gamma(m \bm v)^3} v_y v_z \\ 
0 & -\frac{v_y}{\gamma(m \bm v)} & 0 & 0 \\
0 & 0 & -\frac{v_y}{\gamma(m \bm v)}  & 0 \\
0 & 0 & 0 & -\frac{v_y}{\gamma(m \bm v)} 
\end{bmatrix} \\
\bm A_y = \begin{bmatrix}
-\frac{v_z}{\gamma(m \bm v)} & \frac{n}{c^2 \, \gamma(m \bm v)^3} v_z v_x &  \frac{n}{c^2 \, \gamma(m \bm v)^3} v_z v_y & - \frac{n}{\gamma(m \bm v)}  + \frac{n}{c^2 \, \gamma(m \bm v)^3} v_z v_z \\ 
0 & -\frac{v_z}{\gamma(m \bm v)} & 0 & 0 \\
0 & 0 & -\frac{v_z}{\gamma(m \bm v)}  & 0 \\
0 & 0 & 0 & -\frac{v_z}{\gamma(m \bm v)} 
\end{bmatrix}
\end{align}
The eigenvalues of the flux Jacobian matrices represent the fastest speed of the relativistic system. Due to the upper-triangular form, the only eigenvalues of the matrices $\bm A_x$, $\bm A_y$, $\bm A_z$ are $\frac{v_x}{\gamma(m \bm v)}$, $\frac{v_y}{\gamma(m \bm v)}$, $\frac{v_z}{\gamma(m \bm v)}$, respectively. In terms of conserved variables, these correspond to $\frac{M_x}{\rho \gamma(m \bm M/ \rho)}$, $\frac{M_y}{\rho \gamma(m \bm M/\rho)}$, $\frac{M_z}{\rho \gamma(m \bm M / \rho)}$. 

\section{Momentum equation} \label{appendix_B}
In this section, we show the derivation of the momentum equation \eqref{eq:weak_M}.
We start with the momentum equation \eqref{eq:deg_zero_hybrid_Mk_relativ} and rewrite it in terms of the center of the distribution $\bm v(\bm x, t)$, so that $\bm M = \rho \bm v$ and $\bm w = \bm v /\gamma(m \bm v) $. 
\begin{align}
\partial_t (\rho \bm v) = - \nabla \cdot (\rho \, \bm w \otimes \bm v) 
\end{align}
where we omitted the Lorentz force term which is simply a source term that can be added and carried through the derivation.
Expanding the right-hand side, we obtain
\begin{align} \label{eq:momentum_eq_simpl}
\partial_t (\rho \bm v) = - \bm v \, \nabla \cdot (\rho \, \bm w) - \rho \, \bm w \cdot \nabla \bm v 
\end{align}
Next, we rewrite the term $\rho \bm w \cdot \nabla \bm v$ 
\begin{proposition} \label{prop:main_identity}
The following identity holds
\begin{align} \label{eq:main_identity}
\rho \bm w \cdot \nabla \bm v = \nabla (\rho \bm w \cdot \bm v) - \bm v \times \nabla \times (\rho \bm w) - \rho \, \nabla \left( \gamma(m \bm v) \right) c^2 - (\bm v \cdot \nabla \rho) \bm w - \rho \bm v \cdot \nabla \left[ \frac{1}{\gamma(m \bm v)} \right] \otimes \bm v
\end{align}
where $\bm w = \bm v / \gamma(m \bm v)$.
\end{proposition}
\begin{proof}
In order to verify this identity, we rewrite the first two terms on the right-hand side \eqref{eq:main_identity}
\begin{align}
\nabla \left( \rho \bm w \cdot \bm v \right) = \nabla(\rho \bm w) \cdot \bm v +\nabla \bm v \cdot (\rho \bm w )\\
\bm v \times \nabla \times \left( \rho \bm w\right) = \nabla \left( \rho \bm w \right) \cdot \bm v - \bm v \cdot \nabla \left( \rho \bm w\right)
\end{align}
substituting into \eqref{eq:main_identity} we obtain
\begin{align} \label{eq:main_identity_simpl}
\rho \bm w \cdot \nabla \bm v = \nabla \bm v \cdot (\rho \bm w ) +  \bm v \cdot \nabla \left( \rho \bm w\right) - \rho \, \nabla \left( \gamma(m \bm v) \right) c^2 - (\bm v \cdot \nabla \rho) \bm w - \rho \bm v \cdot \nabla \left[ \frac{1}{\gamma(m \bm v)} \right] \otimes \bm v
\end{align}
next we rewrite the second and the third terms on the right-hand side of \eqref{eq:main_identity_simpl}
\begin{align}
& \bm v \cdot \nabla (\rho \bm w) = \bm v \cdot \nabla \left( \rho \frac{\bm v}{\gamma(m \bm v)}\right)  = (\bm v \cdot \nabla \rho ) \bm w + \rho \bm w \cdot \nabla \bm v + \rho \bm v \cdot \nabla \left[ \frac{1}{\gamma(m \bm v)} \right] \otimes \bm v \\ \label{eq:interm_identity}
& \rho \nabla(\gamma(m \bm v)) c^2 = \rho \frac{\nabla \bm v \cdot \bm v}{\gamma(m \bm v)} = \nabla \bm v \cdot (\rho \bm w)
\end{align}
substituting into \eqref{eq:main_identity_simpl} gives the desired result.
\end{proof}
Using the identity from Proposition \ref{prop:main_identity}, we rewrite the momentum equation \eqref{eq:momentum_eq_simpl} and test it with a smooth test function $\bm \mu$
\begin{align} \label{eq:momentum_eq_weak}
\int_{\Omega} \partial_t (\rho \bm v) \cdot \bm \mu \, dx &= - \int_{\Omega} (\bm \mu \cdot \bm v ) \, \nabla \cdot (\rho \, \bm w) \, dx \\ \notag
& - \int_{\Omega} \bm \mu \cdot \nabla (\rho \bm w \cdot \bm v) \, dx + \int_{\Omega} \left[ \bm v \times \nabla \times (\rho \bm w) \right] \cdot \bm \mu \, dx + \int_{\Omega} \rho \, \bm \mu \cdot \nabla \left( \gamma(m \bm v) \right) c^2 \, dx \\ \notag
& + \int_{\Omega} (\bm v \cdot \nabla \rho) (\bm w \cdot \bm \mu) \, dx + \int_{\Omega} \rho \bm v \cdot \nabla \left[ \frac{1}{\gamma(m \bm v)} \right] ( \bm v \cdot \bm \mu ) \, dx 
\end{align}
we expand the first term on the right-hand side of \eqref{eq:momentum_eq_weak}
\begin{align}
-\int_{\Omega} (\bm \mu \cdot \bm v) \nabla \cdot (\rho \bm w) \, dx = - \int_{\Omega} (\bm \mu \cdot \bm v) (\nabla \rho \cdot \bm w) \, dx - \int_{\Omega} (\bm \mu \cdot \bm v) \left(\rho \frac{\nabla \cdot \bm v}{\gamma(m \bm v)} \right) \, dx - \int_{\Omega} (\bm \mu \cdot \bm v) \left(\rho \bm v \cdot \nabla \left(\frac{1}{\gamma(m \bm v)} \right) \right)
\end{align}
we expand the second term 
\begin{align}
- \int_{\Omega} \bm \mu \cdot \nabla (\rho \bm w \cdot \bm v) \, dx & = - \int_{\Omega} (\bm \mu \cdot \nabla \rho )(\bm w \cdot \bm v) \, dx - \int_{\Omega} (\bm \mu \cdot \bm v) \cdot (\rho \bm w) \, dx 
\end{align}
we expand the third term 
\begin{align}
\int_{\Omega} \left[ \bm v \times \nabla \times (\rho \bm w) \right] \cdot \bm \mu \, dx &= \int_{\Omega} \left[\bm \mu \cdot \nabla \left( \rho \bm w \right) \cdot \bm v - \bm v \cdot \nabla \left( \rho \bm w \right) \cdot \bm \mu \right] \, dx \\ \notag
&= \int_{\partial \Omega} (\bm \mu \cdot \bm n ) (\rho \bm w \cdot \bm v) \, ds - \int_{\Omega} \left(\nabla \cdot \bm \mu \right) (\rho \bm w \cdot \bm v) \, dx - \int_{\Omega} \left( \bm \mu \cdot \nabla \bm v \right) \cdot (\rho \bm w) \, dx \\ \notag
&- \int_{\partial \Omega} (\bm v \cdot \bm n)(\rho \bm w \cdot \bm \mu)\, ds + \int_{\Omega} (\nabla \cdot \bm v) (\rho \bm w \cdot \bm \mu) \, dx + \int_{\Omega} (\bm v \cdot \nabla \bm \mu ) \cdot (\rho \bm w) \, dx 
\end{align}
where we used the integration by parts in the last step. Then \eqref{eq:momentum_eq_weak} simplifies
\begin{align} \notag
-\int_{\Omega} \partial_t(\rho \bm v) \cdot \bm \mu \, dx =& - \int_{\Omega} \rho \bm \mu \cdot \nabla(\gamma(m \bm v)) c^2 \, dx - \int_{\Omega}(\bm \mu \cdot \nabla \rho) (\bm w \cdot \bm v) \, dx \\ \notag
& - \int_{\Omega} \rho \bm \mu \cdot \nabla \bm w \cdot \bm v \, dx - \int_{\Omega}(\nabla \cdot \bm \mu ) (\rho \bm v \cdot \bm w) \, dx \\ \label{eq:simplified}
& + \int_{\Omega} (\rho \bm v \cdot \nabla \bm \mu \cdot \bm w) \, dx + \int_{\partial \Omega} (\bm \mu \cdot \bm n) (\rho \bm w \cdot \bm v) \, dx - \int_{\partial \Omega} (\bm v \cdot \bm n) (\rho \bm w \cdot \bm \mu) \, dx
\end{align}
in this simplification we also used \eqref{eq:interm_identity}.
Next, we work on the second and the fourth terms on the right-hand side of \eqref{eq:simplified}.
\begin{align}
- \int_{\Omega}(\bm \mu \cdot \nabla \rho) (\bm w \cdot \bm v) \, dx - \int_{\Omega}(\nabla \cdot \bm \mu ) (\rho \bm v \cdot \bm w) \, dx &= -\int_{\Omega} \nabla \cdot (\rho \bm \mu) (\bm v \cdot \bm w) \, dx \\ \notag
&= -\int_{\partial \Omega} (\bm \mu \cdot \bm n) (\rho \bm v \cdot \bm w) \, dx + \int_{\Omega} \rho \bm \mu \cdot \nabla (\bm w \cdot \bm v) \, dx
\end{align}
substituting into \eqref{eq:simplified} we obtain
\begin{align} \notag
\int_{\Omega} \partial_t(\rho \bm v) \cdot \bm \mu \, dx =& \int_{\Omega} (\bm w \cdot \nabla \bm \mu  - \bm \mu \cdot \nabla \bm w) \cdot (\rho \bm v )\, dx - \int_{\Omega} \rho \bm \mu \cdot \nabla \left(\gamma(m \bm v) - \frac{\bm w \cdot \bm v}{c^2} \right) c^2 \, dx \\
& - \int_{\partial \Omega} (\bm v \cdot \bm n) (\rho \bm w \cdot \bm \mu) \, dx
\end{align}
in terms of the conserved variables the boundary term reads $ \int_{\partial \Omega} (\bm v \cdot \bm n) (\rho \bm w \cdot \bm \mu) \, dx =  \int_{\partial \Omega} (\bm w \cdot \bm n) (\rho \bm v \cdot \bm \mu) \, dx  =  \int_{\partial \Omega} (\bm w \cdot \bm n) (\bm M \cdot \bm \mu) \, dx$. The boundary term clearly vanishes for $\bm M \cdot \bm n |_{\partial \Omega}=0$.
\end{appendix}

\end{document}